\documentclass[final]{siamltex}
\usepackage{amsfonts}
\usepackage{bbm}
\usepackage{mathrsfs}
\usepackage{amsmath}
\usepackage{epsfig}
\usepackage{graphicx}
\usepackage{subfigure}
\usepackage{epstopdf,comment}
\newtheorem{remark}[theorem]{Remark}

\title{Long time $\mathcal H^s_{\alpha}$ stability of a classical scheme for Cahn-Hilliard equation with polynomial nonlinearity
\thanks{This
work was supported by the National Natural Science Foundation of China (Grant No. 11771060).}
}

\author{Wansheng Wang\thanks{Department of Mathematics, Shanghai Normal University, Shanghai, 200234, China ({\tt w.s.wang@163.com}).}
        }

\begin{document}

\maketitle

\begin{abstract}
In this paper we investigate the long time stability of the implicit Euler scheme for the Cahn-Hilliard equation with polynomial nonlinearity. The uniform estimates in $H^{-1}$ and $\mathcal H^s_\alpha$ ($s=1,2,3$) spaces independent of the initial data and time discrete step-sizes are derived for the numerical solution produced by this classical scheme with variable time step-sizes. %Such $\mathcal H^s_\alpha$ ($s=2,3$) estimates are particularly important since Cahn-Hilliard equation is a fourth-order equation, and it is difficult to obtain such estimates for numerical solution.
The uniform $\mathcal H^3_\alpha$ bound is obtained on basis of the uniform $H^1$ estimate for the discrete chemical potential which is derived with the aid of the uniform discrete Gronwall lemma. A comparison with the estimates for the continuous-in-time dynamical system reveals that the classical implicit Euler method can completely preserve the long time behaviour of the underlying system. Such a long time behaviour is also demonstrated by the numerical experiments with the help of Fourier pseudospectral space approximation.
%These results are not seen in previous literature.
\end{abstract}



\pagestyle{myheadings} \thispagestyle{plain} \markboth{W. S.
WANG}{Long time stability of IEM scheme for Cahn-Hilliard equation}

\section{Introduction} This paper addresses the long time behavior of the implicit Euler scheme for the Cahn-Hilliard equation:
\begin{eqnarray}
\frac{\partial u}{\partial t} &=&
\Delta (-\varepsilon \Delta u+f(u)),\qquad x \in \Omega,\qquad t>
0,\label{eq1.1}\\
u(x,0)&=&u_{0}(x),\qquad x \in \Omega,\label{eq1.2}
\end{eqnarray}
where $\Omega$ is a bounded domain in ${\mathbb R}^{d}~(d \le 3)$ with a sufficiently smooth boundary $\partial\Omega$, $\varepsilon > 0$ is a phenomenological constant modeling the effect of interfacial energy. The original $1$D Cahn-Hilliard equation has the nonlinear term $f(u)=-au+u^3$ with constant $a>0$. In this paper, we consider more general expressions for $f$ which allows polynomial of any odd degree with a positive leading coefficient when $d\le 2$. This will be addressed in the next section.

The differential equation (\ref{eq1.1}) was initially introduced by Cahn and Hilliard \cite{Cahn58} as a model equation for describing the dynamics of pattern formation in phase transition, which was phenomenologically observed in phase separation of a binary solution under sufficient cooling. This kind of pattern formation has been observed in alloys, glasses, polymer solutions, and liquid mixtures (see \cite{Novick-Cohen84} and references therein). Pattern formation may becomes very complicated in long time behaviour. Understanding and predicting the asymptotic behaviour of systems is a fundamental issue \cite{Zheng86,Novick-Cohen08}. It is well known that the continuous dynamical system generated by Cahn-Hilliard equation is dissipative (in the sense of possessing a compact global attractor) and posses very complex behavior with various attractors (see, e.g., \cite{Temam88,Constantin89,Dlotko94,Li98,Sell02,Song09}).

From a numerical point of view, it is important and challenging to study the potential of numerical methods in capturing
the long time behaviour of the underlying system (see, e.g., \cite{Kloeden86,Hale88,Stuart96}). It has been discovered recently that if the dissipativity of a dissipative system is preserved appropriately, then the numerical scheme would be able to capture the long time statistical property of the underlying dissipative system asymptotically, in the sense that the invariant measures of the scheme would converge to those of the continuous-in-time system (see, e.g., \cite{Chekroun12,WangX10,WangX17}). The convergence of exponential attractors also depends on the dissipativity of the scheme \cite{Pierre18}. As a consequence, the long-time stability and dissipativity of numerical schemes have been discussed for various types of equations such as ODEs (see, e.g., \cite{Stuart96}),  Volterra functional differential equations (see, e.g., \cite{Wen11}), neutral delay differential equations \cite{Wang12,Chekroun12,Wang172}, $2$D Navier-Stokes equations (see, e.g., \cite{Ju02,Tone06,WangX11,Gottlieb12,Cheng16}), the infinite Prandtl number model \cite{Cheng08,WangX10}, the $2$D magnetohydrodynamics equations \cite{Tone09}, the $2$D Rayleigh-Benard convection problem \cite{Tone11}, generalized Allen-Cahn equation \cite{Pierre18}, the $2$D thermohydraulics equations \cite{Tone12}, incompressible two-phase flow model \cite{Medjo13}, the $2$D double-diffusive convection \cite{Tone15}, Stokes-Darcy system \cite{Chen13a,Chen16}, Navier-Stokes equations with delay \cite{Wang17}, general semilinear parabolic equations \cite{Elliott93}, and so on. We note that all these PDEs studied in the above literature are second-order equations. For the Cahn-Hilliard equation (1.1), a fourth-order equation, many results on the long time energy stability or the energy decay property of numerical solutions have been reported in the literature (see, e.g., \cite{Elliott92,Choo98,Xia07,He091,Wise09,Hu09,Shen10,Zhang10,Boyer10,Eyre11,Li17,Xu18}). Especially, the scalar auxiliary variable (SAV) approach has been recently introduced for solving gradient flows (see, e.g., \cite{Shen18,Shen181,Li19,Shen19,Akrivis19}). To our best knowledge, however, few studies have been done on the long time $\mathcal H^s_\alpha$ ($s=2,~3$) stability or dissipativity. It is worth emphasizing that solving numerically Cahn-Hilliard equation is an active research area. Different numerical
schemes have been proposed for solving this class of equation, including finite difference method (see, e.g., \cite{Sun95,Choo98,Furihata01,Hu09,Wise09,Wise10,Li12}), finite element method (see, e.g., \cite{Elliott86,Elliott87,Du91,Feng04}), mixed finite element method (see, e.g., \cite{Elliott89,Elliott92,Feng041,Wang16,Diegel16}), nonconforming element method (see, e.g., \cite{Elliott891,Zhang10}), discontinuous Galerkin method (see, e.g., \cite{Choo05,Wells06,Feng07,Kay07,Xia07}), spectral method (see, e.g., \cite{He091,Shen10,Chen13,Li17}), surface finite element method (see, e.g., \cite{Elliott15}), postprocessing mixed finite element method (see, e.g., \cite{Wang16,Zhou15}), least squares spectral element method \cite{Fernandino11}, Multigrid method (see, e.g., \cite{Kim04,Hu09,Wise10}). Stabilization or convex splitting schemes are also investigated in \cite{He07,Wise09,Shen10,Eyre11,Li17,Xu18,Chen19}.

The purpose of this work is to study the long time stability of numerical methods for Cahn-Hilliard equation (1.1) and establish uniform estimates, independent of the initial data $u_0$ and time discrete step-sizes, for the numerical solutions in $\mathcal H^s_\alpha$ ($s=1,2,3$, see their definition in Section 2). The uniform in $\mathcal H^s_\alpha$ ($s=1,2,3$) estimates will guarantee the uniform dissipativity in $\mathcal H^{s-1}_\alpha$  for boundedness in $\mathcal H^s_\alpha$ implies pre-compactness in $\mathcal H^{s-1}_\alpha$  by the Rellich compactness theorem. Since this is the first work on the long time $\mathcal H^s_\alpha$ ($s=2,~3$) stability analysis for such a fourth order equation, the classical implicit Euler scheme becomes the first object of our research. We also note that the importance of using fully implicit schemes for solving Cahn-Hilliard equation has been demonstrated in \cite{Xu18}.

To accomplish these we first make some assumptions on the equation and include some preliminaries on the functional setting in Section 2. Then we review the results on the uniform a priori estimates for the solution to Cahn-Hilliard equation (1.1) in the affine space $\mathcal H^s_\alpha$ in Section 3. In Section 4 we first show unconditionally long-time stability of the classical implicit Euler scheme in space $H^{-1}$ for any time step sizes, and establish uniform $H^{-1}$ estimate independent of the initial data $u_0$ and time discrete step-sizes for the time semi-discrete solution to Cahn-Hilliard equation. Then we obtain uniform $\mathcal H^s_\alpha$ ($s=1,2$) estimates for the numerical solution in subsections 4.3 and 4.4. The uniform $\mathcal H^3_\alpha$ estimate is also established for the first time, based on the uniform $H^1$ estimates for the discrete chemical potential which is obtained with the aid of the uniform discrete Gronwall lemma in this section. The uniform estimates in Sobolev spaces of high order are important for fourth order equations; this is new feature of our work. As a consequence of these estimates, for every time grid, we build a global attractor of the discrete-in-time dynamical system. With space discretization by Fourier pseudospectral methods, we present numerical results in Section 5 that illustrate such a long time stability of this classical scheme. We close by providing some concluding remarks in Section 6.

\section{Preliminaries: Assumption and notations}  In this section we will  make some assumptions on the equation and introduce some function spaces.

\subsection{The Cahn-Hilliard equation} We consider the equation (\ref{eq1.1}) associated with boundary conditions which could be one of two types:
\begin{eqnarray}
\frac{\partial u}{\partial \nu} = \frac{\partial}{\partial \nu}(-\varepsilon \Delta u + f(u))=0,\quad &x \in \partial\Omega,&\quad t>0,\label{eq1.4}\\
u(x+Le_i,t)=u(x,t),\quad &x \in \partial\Omega,&\quad t>0,\quad i=1,\cdots,d,\label{eqa2.2}
\end{eqnarray}
where $\nu$ is the outward unit normal vector along $\partial\Omega$. The Neumann boundary condition (\ref{eq1.4}) is sometimes called a flux boundary condition. In the case of periodic boundary condition (\ref{eqa2.2}), we understand $\Omega$ to be a cube $(0,L)^d$ with $e_1,\cdots,e_d$ being the canonical basis of $\mathbb R^d$. With these two boundary conditions, the
conservation of mass property of a sufficiently smooth solution of (\ref{eq1.1})-(\ref{eq1.2})
\begin{eqnarray}\label{eqa2.3} \int_\Omega u(x,t)dx=\int_\Omega u_0(x)dx,\qquad
t\ge 0,\end{eqnarray} follows immediately from
\begin{eqnarray*} \frac{\partial}{\partial t}\int_\Omega
u(x,t)dx&=&0.\end{eqnarray*}
The Cahn-Hilliard equation (\ref{eq1.1}) is frequently referred to as the $H^{-1}$ gradient flow:
$$u_t=-{\hbox{grad}}_0\mathcal E(u),$$
where the symbol  {\textquotedblleft ${\hbox{grad}}_0$\textquotedblright}
denotes a constrained gradient in a Hilbert space, defined by
$\int_\Omega u \, dx=\hbox{constant}$, and $\mathcal E(u)$ denotes the
Ginzburg-Landau free energy
$$\mathcal E(u)=\int_\Omega\left(\frac{\varepsilon}{2}|\nabla u|^2+F(u)\right)dx=\frac{\varepsilon}{2}| u|^2_1+\int_\Omega F(u)dx.$$

On the derivative $f$ of the potential function $F$ we make following general assumptions:
\begin{enumerate}
\item [(i)] There exist two constants $c_0>0$ and $c_{1}$ such that
\begin{eqnarray}\label{eqa2.4}
f(v)v \geq pc_0v^{2p} - c_{1},\qquad \forall v\in \mathbb R,
\end{eqnarray}
where $2\le p<\infty$ if $d\le 2$, $p=2$ if $d=3$;
\item [(ii)] For every $\eta > 0$, there exists a constant $c_{2} = c_{2}(\eta)$ such that
\begin{eqnarray}\label{eqa2.5}
|f(v)| \leq \eta c_0v^{2p} + c_{2},\qquad \forall v\in \mathbb R;
\end{eqnarray}
\item [(iii)] The primitive $F$ of $f$ vanishes at $v = 0$, and there exist two constants $c_{3}$ and $c>0$ such that
\begin{eqnarray}\label{eqa2.6}
\frac{1}{2}c_0v^{2p} - c_{3} \leq F(v) &\leq& \frac{3}{2}c_0v^{2p} + c_{3},\qquad \forall v\in \mathbb R;\\
F^{\prime\prime}(v)&\ge& -c, \qquad \forall v\in \mathbb R.\label{eqa2.7}
\end{eqnarray}
\end{enumerate}

\begin{remark} It is typical to assume that $f$ is a polynomial function of the degree $2p-1$ with a positive leading coefficient, namely, (see, e.g., \cite{Temam88,Sell02,Boyer10}),
\begin{eqnarray}\label{eqa2.8}
f(v)=\sum_{j=1}^{2p-1}a_{j}v^{j}, \qquad a_{2p-1}>0,
\end{eqnarray}
where $2\le p<\infty$ if $d\le 2$, $p=2$ if $d=3$. Obviously, the function $f$ defined by (\ref{eqa2.8}) satisfies the assumptions (i) and (ii). Since the primitive $F$ of $f$ vanishes at $v = 0$, we have
\begin{eqnarray}
F(v) = \sum_{j=2}^{2p}b_{j}v^{j},\qquad jb_{j} = a_{j-1},\qquad 2\leq j\leq 2p.
\end{eqnarray}
with $b_{2p}>0$. Then it is easy to verify that the assumption (iii) is satisfied.
\end{remark}

\subsection{Function spaces}
In order to study the long time dynamics of the Cahn-Hilliard equation, we introduce some function spaces. Let $H^s=H^{s}(\Omega)$, $s\ge 0$, be the standard Sobolev
space with the norm $\|\cdot\|_{s}$, and let $\|\cdot\|$ and $(\cdot,\cdot)$ denote the usual norm and inner product in $H=L^2=L^2(\Omega)$. In addition, define for $s\ge 0$,
\begin{eqnarray*}
H^{-s}=(H^s)^*,\qquad H^{-s}_0=\left\{v\in H^{-s},~\langle v,1\rangle=0\right\},
\end{eqnarray*}
where $\langle\cdot,\cdot\rangle$ stands for the dual product between $H^s$ and $H^{-s}$. In view of the conservation of mass property (\ref{eqa2.3}), we denote by $m(\varphi)$ the average on $\Omega$ of a function $\varphi$ in $L^2(\Omega)$ (or $L^1(\Omega)$)
\begin{eqnarray}
m(\varphi) = \frac{1}{|\Omega|}\int_{\Omega}\varphi(x)dx,\label{eq2.1}
\end{eqnarray}
and we write $\overline \varphi = \varphi - m(\varphi)$. Then for $s\ge 0$ we define $H^s_0$ and $H^s_\beta$, $\beta\in \mathbb R$, by
\begin{eqnarray}\label{eq2.11}
H^s_\beta=\left\{\varphi\in H^s,~~m(\varphi)=\beta\right\}.
\end{eqnarray}
Notice that the set $H^s_0$ is a closed linear subspace of $H^s$, and for each $\beta\in \mathbb R$ with $\beta\not=0$, $H^s_\beta$ is a hyperplane in $H^s$. For any $\alpha\ge 0$ and any compact set $K$ in $\mathbb R$, we define $\mathcal H^s_\alpha$ and $\mathcal H^s_K$ by
\begin{eqnarray}\label{eq2.12}
\mathcal H^s_\alpha=\bigcup\limits_{|\beta|\le \alpha}H^s_\beta,\qquad \mathcal H^s_K=\bigcup\limits_{\beta\in K}H^s_\beta.
\end{eqnarray}
It is obvious that $H^s=\mathcal H^s_{\mathbb R}$.

In order to put the above two initial-boundary value problems in a common abstract framework, we define the linear operator $A=-\Delta$ with domain of
definition
\begin{eqnarray*}
\mathcal D(A)&=&\left\{v\in H^2: \frac{\partial v}{\partial \nu}=0~{\hbox{on}} ~\partial
\Omega\right\},\\
\mathcal D(A)&=&\left\{v\in H^2: v(x+Le_i)=v(x)~{\hbox{for}} ~x\in\partial
\Omega,~~i=1,\cdots,d\right\},
\end{eqnarray*}
for the two sets of boundary conditions, respectively. Since $A$ is a self-adjoint positive semidefinite and
densely defined operator on $H$, for
real $s$, we can define the spaces
$\dot{H}^s=\mathcal D(A^{s/2})$ with norms $|v|_s=\|A^{s/2}v\|$. It is well known that, for integer
$s\ge 0$, $\dot{H}^s$ is a subspace of $H^s$ and that the
norms $|\cdot|_s$ and $\|\cdot\|_s$ are equivalent on $\dot{H}^s$. Using the above notations, we may write (\ref{eq1.1})-(\ref{eq1.2})
as an abstract initial value problem \begin{eqnarray}
\label{eq2.4}&& u_t+\varepsilon A^2
u+Af(u)=0,\qquad t>0,\\
\label{eq2.5}&&u(0)=u_0.
\end{eqnarray}

We define $G: H\to \dot H^2$ as the inverse of $A$. It can be easily verified that $G$ is self-adjoint and positive
semidefinite on $H$. Clearly,
\begin{eqnarray}|v|_{-1}=\|G^{\frac{1}{2}}v\|=\sup\limits_{\chi\in \dot{H}^1}\frac{|\langle v,\chi\rangle|}{|v|_1},\quad \forall v\in \dot{H}^1,
\end{eqnarray} and $|v|_{-1}$ is a continuous norm on $L^2(\Omega)$; see \cite{Temam88}. As a result, there exists a constant $\gamma_1>0$ depending only on $\Omega$ such that \cite{Temam88}
\begin{eqnarray}
|\bar v|_{-1}=|v - m(v)|_{-1}\leq \gamma_1| v|_1,\qquad ~~~~~~\forall v \in H^{1}.\label{eq2.16}
\end{eqnarray}
Applying $G$ to (\ref{eq2.4}), we have
\begin{eqnarray}
Gu_t+\varepsilon A
u+f(u)=0,\qquad t>0.\label{eq2.18}
\end{eqnarray}

\section{Long-time dynamics of continuous system}  For the initial value problem (\ref{eq2.4})-(\ref{eq2.5}), the mapping  $S(t):~u_0\in H\to u(t)\in H$ satisfies the semigroup property $S(t)S(s)=S(t+s), ~t,s\ge 0$, $S(0)=I$, and the pair $(H,S(t))$ is a dynamical system. We first note that the semigroup $S(t)$ cannot have a global attractor in $H=L^2$, since the set $Q$ of stationary solutions is unbounded. In deed, the property that the average of $u$ is conserved excludes the existence of an absorbing set in $H$. Nevertheless, this semigroup $S(t)$ possesses a host of attractors with many interesting dynamical properties.

Following the approach of Elliott and Larsson \cite{Elliott92}, we can get
\begin{eqnarray}\label{eq2.17} \int_0^t|u_t|^2_{-1}ds+\mathcal E(u(t))=\mathcal E(u_0),\qquad t\ge 0,\end{eqnarray} which implies that the total energy is
nonincreasing in time, provided that $u_0\in \dot{H}^{1}$. $\mathcal E(u)$ is thus known as a Lyapunov functional for the initial value problem (\ref{eq2.4}). The chemical potential $\omega$ is the derivative of $\mathcal E$, i.e.,
$$\omega=\mathcal E^\prime (u)=f(u)-\varepsilon\Delta u=-Gu_t.$$ From (\ref{eq2.17}), we also obtain an a priori bound:
\begin{eqnarray}\label{eqa3.2}|u(t)|_1\le C(R_1),\quad t\ge 0,
\end{eqnarray} provided that $u_0\in \dot{H}^{1}$ with $|u_0|_1\le R_1$.

Based on this bound, the following result about the uniform bound in $H^{-1}$ has been shown in \cite{Temam88}.

\begin{proposition}[Uniform estimate in $H^{-1}$, \cite{Temam88}]
Assume that (\ref{eqa2.4}), (\ref{eqa2.5}), and (\ref{eqa2.6}) are fulfilled, and
\begin{eqnarray}\label{eq2.18a}
|\overline{u_0}|_{-1}\le R\qquad {\hbox{and}}\qquad |m(u_0)|\le \alpha.
\end{eqnarray}
Let $\rho_0$ be defined by
\begin{eqnarray}\label{eq2.19}\qquad \rho_0:=\gamma_1\left(\frac{2C_1(\alpha)}{\varepsilon}\right)^{\frac{1}{2}}\qquad {\hbox{with}} \qquad C_1(\alpha)=(c_1+c_2\alpha+2c_3)|\Omega|.
\end{eqnarray}
Then there exists a time $t_0=t_0(R,\alpha,\rho_0)$ such that
\begin{eqnarray}\label{eqa3.4}
|\bar u(t)|_{-1}\le \rho_0^\prime:=2\rho_0,\qquad \forall t\ge t_0.
\end{eqnarray}

\end{proposition}

This result shows the existence of an absorbing set for $S(t)$ on the affine space $H_\alpha$ endowed with the norm $|\cdot|_{-1}$ (see, e.g., \cite{Temam88}). The existence of an absorbing set is related to a dissipativity property for the dynamical system. Furthermore, the following result guarantees the existence of an absorbing set in $\mathcal H_\alpha$ and in $\mathcal H^1_\alpha$ and the existence of a global attractor $\mathcal A_\alpha$.
\begin{proposition}[Uniform estimates in $\mathcal H^1_\alpha$, \cite{Temam88,Li98}]
Assume that (\ref{eqa2.4}), (\ref{eqa2.5}), (\ref{eqa2.6}), (\ref{eqa2.7}), and (\ref{eq2.18a}) are fulfilled. Let $\rho_0$ and $\rho_0^\prime$ be defined by (\ref{eq2.19}) and (\ref{eqa3.4}), respectively. Then there exists a time $t_1=t_0+r$ with $r>0$ arbitrary such that
\begin{eqnarray}\label{eqa3.6}
|u(t)|_1\le \rho_1,\qquad \forall t\ge t_1,
\end{eqnarray}
where $\rho_1=C^{\frac{1}{2}}_2(\alpha)$ with $C_2(\alpha)=\frac{2 (\rho_0^\prime)^2}{r\varepsilon}+\frac{2C_1(\alpha)}{\varepsilon}$. Furthermore, let $\mathcal H_\alpha$ be defined by (\ref{eq2.12}). Then for every $\alpha\ge 0$, the semigroup $S(t)$ associated with (\ref{eq2.4})-(\ref{eq2.5}) maps $\mathcal H_\alpha$ into itself and possesses in $\mathcal H_\alpha$ and $\mathcal H^1_\alpha$ a maximal attractor $\mathcal A_\alpha$ that is compact and connected.
\end{proposition}

In the original proof of the dissipativity estimates (\ref{eqa3.6}) in \cite{Temam88}, the condition (\ref{eqa2.7}) is replaced by
\begin{eqnarray}\label{eq3.7}F^{\prime\prime}(v)\ge c_4v^{2p-2} -c, \qquad \forall v\in \mathbb R,\end{eqnarray}
with constant $c_4>0$.

The following uniform estimates in affine space $\mathcal H^s_\alpha$, $s=2,3$, have been obtained in \cite{Li98} under some slightly weaker conditions than (\ref{eqa2.4}), (\ref{eqa2.5}), (\ref{eqa2.6}) and (\ref{eqa2.7}).

\begin{proposition}[Uniform estimates in $\mathcal H^s_\alpha$, $s=2,3$, \cite{Li98}]
Assume that (\ref{eqa2.4}), (\ref{eqa2.5}), (\ref{eqa2.6}), (\ref{eqa2.7}), and (\ref{eq2.18a}) are fulfilled. For any $R,\alpha>0$, there exist positive continuous function $\mu_s(t)$ defined on $(0,\infty)$ and positive constants $\rho_s,~t_s>0$ such that
 \begin{eqnarray}
|u(t)|_s&\le& \mu_s(t),\qquad \forall t>0,\\
|u(t)|_s&\le& \rho_s,\qquad \forall t\ge t_s,\label{eq3.9}
\end{eqnarray}
where $\mu_s$ and $t_s$ depends on $R,~\alpha,~\varepsilon$ and $F$; $\rho_s$ depends only on $\alpha,~\varepsilon$ and $F$. Furthermore, let $\mathcal H^s_\alpha$ be defined by (\ref{eq2.12}). Then for every $\alpha\ge 0$, the semigroup $S(t)$ associated with (\ref{eq2.4})-(\ref{eq2.5}) maps $\mathcal H^s_\alpha$ into itself and possesses in $\mathcal H^s_\alpha$ a maximal attractor $\mathcal A_\alpha$ that is bounded and connected in $\mathcal H^s_\alpha$.
\end{proposition}

It should be pointed out that the uniform estimate (\ref{eq3.9}) in $\mathcal H^2_\alpha$ has been obtained in \cite{Temam88} under the conditions (\ref{eqa2.4}), (\ref{eqa2.5}), (\ref{eqa2.6}), and (\ref{eq3.7}).

From the above uniform a priori estimates, we know there exists a global attractor for continuous-in-time system semigroup $S(t)$. %and hence stationary statistical solution or invariant probability measures, associated with temporal averages via generalized limits (see, for example, \cite{Chekroun12,Lax02}). %The generalized limits are known as bounded linear functionals on the space of bounded functions that agree with the usual limit on those functions whenever the usual limit exists.
An accurate numerical approximation of the Cahn-Hilliard equation should mimic its long-term behavior. In the next section, we will investigate the uniform bound of the numerical solution produced by the implicit Euler method.

\section{Time uniform bounds for the semi-discrete scheme}

In this section, we discuss the time discretization of (1.1) by fully implicit Euler method and derive uniform estimates for the numerical solution in $H^{-1}$ and $\mathcal H^s_\alpha$ ($s=1,2,3$). To obtain the uniform estimate in space $\mathcal H^3_\alpha$, we need estimate the discrete chemical potential $\omega^{n+1}$ in $H^1$.

Let $\mathcal J: 0=t^0<t^1<\cdots<t^n<\cdots$ be a partition of $[0,\infty)$, $I_n := [t^{n},t^{n+1}]$, $k_n:=t^{n+1} - t^{n}$, and $k=\sup_{n\ge 0} k_n$.  Then a semi-discrete formulation of (\ref{eq2.4}) via implicit Euler method on the time mesh $\mathcal J$ reads:
\begin{eqnarray}\label{eq3.4a}\bar\partial_{t}u^{n+1}+\varepsilon A^2u^{n+1}+Af(u^{n+1})=0, \qquad n\ge 0,
\end{eqnarray}
with $u^0=u_0$, where $\bar\partial_t u^{n+1}=(u^{n+1}-u^n)/k_n$. To explicitly approximate the chemical potential $\omega$, we may equivalently write (\ref{eq3.4a}) as
\begin{eqnarray}
\bar\partial_{t}u^{n+1}+A w^{n+1} &=& 0,\label{eq3.4}\\
w^{n+1}&=&\varepsilon A u^{n+1} + f(u^{n+1}).\quad \label{eq3.5}
\end{eqnarray}
It is easy to verify that the conservation of mass property is preserved by implicit Euler scheme (\ref{eq3.4a}):
\begin{eqnarray}
\int_{\Omega}\bar\partial_t u^{n+1}dx = 0,\qquad \int_{\Omega}u^{n+1}dx = \int_{\Omega}u_0(x)dx,\quad \forall n\geq 0.\label{eq3.6}
\end{eqnarray}

We shall study the long time stability of numerical scheme (\ref{eq3.4a}) and obtain uniform bounds necessary for the convergence of the attractor and associated invariant measures of the discretised system to those of the continuous system. For this purpose, we view the scheme as a mapping on $\dot H^s$:
\begin{eqnarray}S_{k}u^{n}=u^{n+1},\qquad n\ge 0.
\end{eqnarray}
%and recall the definition of uniform dissipativity.
%\begin{definition}[Uniform dissipativity, \cite{WangX10}] Let $\{S_k,0<k\le k_0\}$ be a family of continuous maps on a Banach space $V$ that generates a family of discrete dissipative dynamical systems (with global attractor $\mathcal A_k$) on $V$. If there exists a $k_1\in (0,k_0)$ such that
%$$K=\bigcup\limits_{0<k\le k_1}\mathcal A_k$$
%is pre-compact in $V$, then $\{S_k,0<k\le k_1\}$ is called uniformly dissipative.
%\end{definition}
%Notice that if $K$ is bounded in $V$, then $\{S_k,0<k\le k_1\}$ is called uniformly bounded.

\subsection{Energy decay}\label{3}

In this subsection, we discuss the discrete analogue of the property of energy decay. It is of some interest to note that there has been a lot of work studying on energy decay property of numerical schemes in the literature (see, e.g., \cite{Elliott92,Choo98,Xia07,He091,Wise09,Hu09,Shen10,Zhang10,Boyer10,Eyre11,Li17}). This property of the fully discrete approximation based on implicit Euler scheme together with finite element methods has also been established in \cite{Elliott92} for $\varepsilon=1$. Since this property will be used in our uniform estimates in $\mathcal H^1_\alpha$, we state this property here.

\begin{proposition}\label{pro4.1} Suppose that (\ref{eqa2.7}) is satisfied. Then if
\begin{eqnarray}\label{eqa4.12}
k\le \frac{8\varepsilon}{c^2},
\end{eqnarray}
we have
\begin{eqnarray}\label{eq4.12}
\mathcal E(u^{n+1})\le \mathcal E(u^{n}),\quad \forall n\ge 0.
\end{eqnarray}
In addition, for any $c_k\in (0,1]$, if
\begin{eqnarray}\label{eq4.9}
k\le \frac{8c_k\varepsilon}{c^2},
\end{eqnarray}
the solution to (\ref{eq3.4})-(\ref{eq3.5}) satisfies
\begin{eqnarray}\label{eq4.10}
(1-c_k)\sum\limits_{j=0}^nk_j\|G^{\frac{1}{2}}\bar\partial_tu^{j+1}\|^2+\mathcal E(u^{n+1})\le \mathcal E(u_{0}),\quad \forall n\ge 0.
\end{eqnarray}

\end{proposition}

\begin{proof} The proof is similar to that of \cite{Elliott92} for full discretization approximation of (\ref{eq1.1}) with $\varepsilon=1$. Applying the operator $G$ to (\ref{eq3.4a}) yields
\begin{eqnarray}\label{eq3.5a}G\bar\partial_{t}u^{n+1}+\varepsilon Au^{n+1}+f(u^{n+1})=0, \qquad n\ge 0.\end{eqnarray}
Taking the inner product of (\ref{eq3.5a}) with $\bar\partial_t u^n$, we obtain, $\forall n\ge 0$,
\begin{eqnarray}\label{eq4.13}
\qquad (G\bar\partial_t u^{n+1},\bar\partial_t u^{n+1})+(\varepsilon A u^{n+1},\bar\partial_t u^{n+1})+(f(u^{n+1}),\bar\partial_t u^{n+1})=0.
\end{eqnarray}
Using the condition (\ref{eqa2.7}), we get
$$f(r)(r-s)\ge F(r)-F(s)-\frac{1}{2}c(r-s)^2,$$
 and therefore the relation
\begin{eqnarray}\label{eqa4.15}
(f(u^{n+1}),\bar\partial_t u^{n+1})\ge \bar\partial_t\int_\Omega F(u^{n+1})dx-\frac{1}{2}ck_n\|\bar\partial_t u^{n+1}\|^2.\end{eqnarray}
As a consequence, we have, from (\ref{eq4.13}),
\begin{eqnarray}\label{eqa4.14}
&&\|G^{\frac{1}{2}}\bar\partial_t u^{n+1}\|^2+\frac{\varepsilon}{2}k_n|\bar\partial_t u^{n+1}|_1^2+\bar\partial_t\mathcal E(u^{n+1})\nonumber\\
&\le& \frac{1}{2}ck_n\|\bar\partial_t u^{n+1}\|^2\le \frac{1}{8\varepsilon}c^2k_n\|G^{\frac{1}{2}}\bar\partial_t u^{n+1}\|^2+\frac{\varepsilon}{2}k_n|\bar\partial_t u^{n+1}|_1^2,
\end{eqnarray}
where we have used
$$(Au^{n+1},\bar\partial_t u^{n+1})=\frac{1}{2}\bar\partial_t|u^{n+1}|_1^2+\frac{1}{2}k_n|\bar\partial_t u^{n+1}|_1^2.$$
Then energy decay inequality (\ref{eq4.12}) follows directly from (\ref{eqa4.14}) if (\ref{eqa4.12}) holds. If $k$ satisfies the condition (\ref{eq4.9}), summing (\ref{eqa4.14}) shows that (\ref{eq4.10}) holds. This completes the proof for the energy decay property of the numerical scheme.
\end{proof}

As a consequence of (\ref{eq4.12}), the a priori bound can be obtained:
If $u_{0}\in \dot H^1$ with $|u_{0}|_1\le R_1$, then
\begin{eqnarray}
|u^n|_1\le C(R_1), \quad n\ge 0,
\end{eqnarray}
which is a discrete counterpart to the a priori bound (\ref{eqa3.2}).

\subsection{Time uniform bound in $H^{-1}$}

In this subsection, we derive a discrete counterpart to the uniform bound (\ref{eqa3.4}).
%For this purpose, we first note that
%\begin{eqnarray}
%\frac{\partial }{\partial t}\overline {u^{n+1}}=\frac{\partial }{\partial t}{\Big(}u^{n+1}-m(u^{n+1}){\Big)}
%=\frac{\partial }{\partial t}u^{n+1}.
%\end{eqnarray}
Take the inner product of (\ref{eq3.4}) with $G\overline {u^{n+1}}$  to obtain
\begin{eqnarray}\label{eq4.14}
\left(\bar\partial_t \overline {u^{n+1}},G\overline {u^{n+1}}\right)+{\Big(} w^{n+1},\overline {u^{n+1}}{\Big)}=0.
\end{eqnarray}
Then we have the following stability estimate in $H^{-1}$.

\begin{lemma}[Stability estimate in $H^{-1}$]\label{lem4.2} Assume that (\ref{eqa2.4}), (\ref{eqa2.5}), (\ref{eqa2.6}), and (\ref{eq2.18a}) are fulfilled. Let $u^{n+1}$ be the solution of the numerical scheme (\ref{eq3.4a}). Then for any time step sequence $\{k_n\}$ with $k=\sup_{n\ge 0}k_n$, we have the following stability estimate, for $n\ge 0$,
\begin{eqnarray}\label{eqa4.19}
|\overline {u^{n+1}}|_{-1}^{2}\leq\exp\left(-\frac{\varepsilon t^{n+1}}{\gamma_1^2+\varepsilon k}\right)|\overline {u_{0}}|_{-1}^{2}
+ 2C_{1}(\alpha)\frac{\gamma_1^2+\varepsilon k}{\varepsilon} \left[1-\exp\left(-\frac{\varepsilon t^{n+1}}{\gamma_1^2+\varepsilon k}\right)\right],~~~
\end{eqnarray}
where $C_1(\alpha)$ has been defined in (\ref{eq2.19}). Furthermore, define
\begin{eqnarray}
\rho_{0k}(\alpha)=\left(2C_1(\alpha)\frac{\gamma_1^2+\varepsilon k}{\varepsilon}\right)^{\frac{1}{2}},
\end{eqnarray}
and let there be given $\rho>\rho_{0k}$. Then there exists a positive integer $n_{0k}=n_{0k}(R,\alpha,\rho,k)$ such that
\begin{eqnarray}\label{eqa4.17}
|\overline {u^{n}}|_{-1}\le \rho,\qquad n\ge n_{0k}.
\end{eqnarray}
\end{lemma}

\begin{proof} By using the definition of the chemical potential $\omega^{n+1}$ and the property of mass conservation, we obtain
\begin{eqnarray}\label{eq4.16}
\left( w^{n+1},\overline {u^{n+1}}\right)&=&-\varepsilon\left(\Delta u^{n+1},\overline {u^{n+1}}\right)+\left(f(u^{n+1}),\overline {u^{n+1}}\right)\nonumber\\
&=&\varepsilon |u^{n+1}|_1^{2}+(f(u^{n+1}),u^{n+1})-(f(u^{n+1}),m(u_{0})).
\end{eqnarray}
Notice $|m(u_{0})|\leq\alpha$. Let $\eta = \frac{1}{\alpha}(p-\frac{3}{2})$ in (\ref{eqa2.5}). The last term on the right hand side of (\ref{eq4.16}) can be bounded as
\begin{eqnarray}\label{eq4.17}
(f(u^{n+1}),m(u_{0}))&=&\int_{\Omega}f(u^{n+1})m(u_{0})dx \nonumber\\
&\leq&\int_{\Omega}[\eta c_0(u^{n+1})^{2p}+c_{2} ]\alpha dx\nonumber\\
&=&\int_{\Omega}\left[\frac{1}{\alpha}\left(p-\frac{3}{2}\right) c_0(u^{n+1})^{2p}+c_{2} \right]\alpha dx.
\end{eqnarray}
Similarly, by (\ref{eqa2.4}), the second term on the right hand side of (\ref{eq4.16}) can be rewritten as
\begin{eqnarray}\label{eq4.18}
(f(u^{n+1}),u^{n+1})&=&\int_{\Omega}f(u^{n+1})u^{n+1}dx\nonumber\\
&\geq&\int_{\Omega}\left[pc_0(u^{n+1})^{2p}-c_{1}\right]dx.
\end{eqnarray}
As a consequence of (\ref{eq4.17}) and (\ref{eq4.18}), we have from (\ref{eq4.16})
\begin{eqnarray}\label{eq4.19}
{\Big(} w^{n+1},\overline {u^{n+1}}{\Big)}&\geq&\int_{\Omega}\frac{3}{2}c_0(u^{n+1})^{2p}dx +
\varepsilon |u^{n+1}|_1^{2} - c_{2}\alpha|\Omega|-c_{1}|\Omega| \nonumber\\
&\geq&\int_{\Omega}\left(\frac{3}{2}c_0(u^{n+1})^{2p} + c_{3}\right)dx - c_{3}|\Omega|\nonumber\\
&&+\varepsilon |u^{n+1}|_1^{2} - c_{2}\alpha|\Omega| - c_{1}|\Omega|\nonumber\\
&\geq&\int_{\Omega}F( u^{n+1})dx + \varepsilon |u^{n+1}|_1^{2} - C_{0}(\alpha)\nonumber\\
&\geq& \mathcal E(u^{n+1}) - C_{0}(\alpha),
\end{eqnarray}
where
\begin{eqnarray*}
C_{0}(\alpha )=[c_{2}\alpha+c_{1}+c_{3}]|\Omega|.
\end{eqnarray*}
Using the identity
\begin{eqnarray}
2(a-b,a) = |a|^2-|b|^2 + |a-b|^2,\label{eq3.16}
\end{eqnarray}
and substituting (\ref{eq4.19}) into (\ref{eq4.14}), we obtain
\begin{eqnarray}\label{eq4.21}
\qquad \frac{1}{2}\left(|\overline {u^{n+1}}|_{-1}^{2}-|\overline {u^{n}}|_{-1}^{2}+|\overline {u^{n+1}}-\overline {u^{n}}|_{-1}^{2}\right)+k_n\mathcal E(u^{n+1})\leq C_{0}(\alpha )k_n.
\end{eqnarray}
In view of the relation (\ref{eq2.16}), we have
\begin{eqnarray}\label{eq4.22}
\frac{\varepsilon}{2}|u^{n+1}|_1^{2}\geq \frac{\varepsilon}{2\gamma^{2}_1}|u^{n+1}-m(u^{n+1})|_{-1}^{2}=\frac{\varepsilon}{2\gamma^{2}_1}|\overline {u^{n+1}}|_{-1}^{2}.
\end{eqnarray}
Combining (\ref{eq4.21}), (\ref{eq4.22}) and (\ref{eqa2.6}) yields
\begin{eqnarray}\label{eq4.23}
\frac{1}{2}\left(|\overline {u^{n+1}}|_{-1}^{2}-|\overline {u^{n}}|_{-1}^{2}\right)+\frac{k_n\varepsilon}{2\gamma^{2}_1}|\overline {u^{n+1}}|_{-1}^{2}\leq C_{1}(\alpha)k_n.
\end{eqnarray}
%where
%\begin{eqnarray*}
%C_{1}(\alpha)=C_{0}(\alpha)+c_3|\Omega|.
%\end{eqnarray*}
%It follows from (\ref{eq4.23}) that
%\begin{eqnarray}
%|\overline {u^{n+1}}|_{-1}^{2}&\leq& \frac{1}{1+\frac{\varepsilon k_n}{\gamma_1^{2}}}|\overline {u^{n}}|_{-1}^{2} + \frac{2k_n}{1+\frac{\varepsilon k_n}{\gamma_1^{2}}}C_{1}(\alpha)\nonumber\\
%&\leq&\prod\limits_{i=0}^n \left(1-\frac{\frac{\varepsilon k_i}{\gamma_1^{2}}}{1+\frac{\varepsilon k_i}{\gamma_1^{2}}}\right)|\overline {u_{0}}|_{-1}^{2}
%+ 2C_{1}(\alpha)\sum_{i=0}^{n}\prod\limits_{j=i}^n \left(1-\frac{\frac{\varepsilon k_j}{\gamma_1^{2}}}{1+\frac{\varepsilon k_j}{\gamma_1^{2}}}\right)k_i.
%\end{eqnarray}
Since $k_i\le k$, $i=0,1,\cdots$, we further have
\begin{eqnarray}
|\overline {u^{n+1}}|_{-1}^{2}&\leq&\prod\limits_{i=0}^n \left(1-\frac{\varepsilon k_i}{\gamma_1^2+\varepsilon k}\right)|\overline {u_{0}}|_{-1}^{2}
+ 2C_{1}(\alpha)\sum_{i=0}^{n}\prod\limits_{j=i}^n \left(1-\frac{\varepsilon k_j}{\gamma_1^2+\varepsilon k}\right)k_i\nonumber\\
&\leq&\prod\limits_{i=0}^n \exp\left(-\frac{\varepsilon k_i}{\gamma_1^2+\varepsilon k}\right)|\overline {u_{0}}|_{-1}^{2}
+ 2C_{1}(\alpha)\sum_{i=0}^{n}\prod\limits_{j=i}^n \exp\left(-\frac{\varepsilon k_j}{\gamma_1^2+\varepsilon k}\right)k_i\nonumber\\
&=&\exp\left(-\frac{\varepsilon t^{n+1}}{\gamma_1^2+\varepsilon k}\right)|\overline {u_{0}}|_{-1}^{2}
+ 2C_{1}(\alpha)\sum_{i=0}^{n} \exp\left(-\frac{\varepsilon (t^{n+1}-t^i)}{\gamma_1^2+\varepsilon k}\right)k_i\nonumber\\
&\leq&\exp\left(-\frac{\varepsilon t^{n+1}}{\gamma_1^2+\varepsilon k}\right)|\overline {u_{0}}|_{-1}^{2}
+ 2C_{1}(\alpha)\int^{t^{n+1}}_0 \exp\left(-\frac{\varepsilon (t^{n+1}-s)}{\gamma_1^2+\varepsilon k}\right)ds,~~~~~
\end{eqnarray}
which implies (\ref{eqa4.19}). For any $u_0$ satisfying $|\overline{u_0}|_{-1}\le R$, for any given $\rho>\rho_{0k}$, there exists a positive integer $n_{0k}$ with
\begin{eqnarray}\label{eq4.27}
t^{n_{0k}}\ge \frac{\gamma_1^2+\varepsilon k}{\varepsilon}\left[\ln R^2-\ln(\rho^2-\rho_{0k}^2)\right]
\end{eqnarray}
such that
\begin{eqnarray*}&&\exp\left(-\frac{\varepsilon t^{n+1}}{\gamma_1^2+\varepsilon k}\right)|\overline {u_{0}}|_{-1}^{2}
+ 2C_{1}(\alpha)\frac{\gamma_1^2+\varepsilon k}{\varepsilon} \left[1-\exp\left(-\frac{\varepsilon t^{n+1}}{\gamma_1^2+\varepsilon k}\right)\right]\le \rho^2,\\
&&\quad \forall n\ge n_{0k}(R,\alpha,\rho,k).
\end{eqnarray*}
Hence, we have (\ref{eqa4.17}). Then we complete the proof.
\qquad\end{proof}

We notice that the stability bounds (\ref{eqa4.19}) and (\ref{eqa4.17}) are unconditional estimates, that is, they hold for any time step sequence $\{k_n\}$. The stability bound (\ref{eqa4.17}) leads to an absorbing set in the affine space $\mathcal H_\alpha$ for $S_{k}$, and hence the mapping $S_{k}$ generates a discrete dissipative dynamical system on $\mathcal H_\alpha$, for each time mesh $\mathcal J$. We wish to point out that $\rho_{0k}\to \rho_0$ as $k\to 0$. This implies that the radius of the absorbing set of discretise system with vanished step sizes is the same as that of continuous system and the implicit Euler method can completely preserve the long-term behavior of the underlying system (1.1) in the affine space $\mathcal H_\alpha$ with the norm $|\cdot|_{-1}$.

With Lemma \ref{lem4.2}, we show the following uniform estimate in $H^{-1}$, independent of the initial data and the time step-sizes $k_n$.

\begin{theorem}[Time uniform bound in $H^{-1}$] Assume that (\ref{eqa2.4}), (\ref{eqa2.5}), (\ref{eqa2.6}), and (\ref{eq2.18a}) are fulfilled. Let $u^{n+1}$ be the solution of the numerical scheme (\ref{eq3.4a}) and define \begin{eqnarray*}
\hat\rho_{0}=\sqrt{2C_1(\alpha)}\left(\frac{\gamma_1^2}{\varepsilon}+\frac{8\varepsilon}{c^2}\right)^{\frac{1}{2}}.
\end{eqnarray*} If the time step sequence $\{k_n\}$ with $k=\sup_{n\ge 0}k_n$ satisfies (\ref{eqa4.12}), we have
\begin{eqnarray}\label{eqa4.21}
|\overline {u^{n}}|_{-1}\le \hat E_0:=\max\left\{R,\hat\rho_{0}\right\},\qquad n\ge 0,\end{eqnarray}
and uniform estimate
\begin{eqnarray}\label{eqa4.22}
|\overline {u^{n}}|_{-1}\le {\hat\rho_0}^\prime:=2\hat\rho_0,\qquad n\ge n_0,
\end{eqnarray}
where $n_0=n_0(R,\alpha,\hat\rho_0)$ will be defined in (\ref{eqa4.27}).
\end{theorem}

\begin{proof} Boundedness estimate (\ref{eqa4.21}) is a direct result of (\ref{eqa4.19}). For any $u_0$ satisfying $|\overline{u_0}|_{-1}\le R$, there exists a positive integer $n_{0}$ with
\begin{eqnarray}\label{eqa4.27}
t^{n_{0}}\ge \left(\frac{\gamma_1^2}{\varepsilon}+\frac{8\varepsilon}{c^2}\right)\ln \left(R^2/\hat\rho^2_0\right)
\end{eqnarray}
such that (\ref{eqa4.22}) holds. The proof is complete.
\qquad\end{proof}

From the inequality (\ref{eq4.21}) and uniform estimate (\ref{eqa4.22}), we also obtain the following $H^{-1}$ bound on the difference of the solution at adjacent time steps:
\begin{eqnarray}
\sum\limits_{j=0}^{n-1}|\overline {u^{j+1}}-\overline {u^{j}}|_{-1}^{2}&\le& 2C_1(\alpha)t^{n}+R^2,\\
\sum\limits_{j=n}^{n+m-1}|u^{j+1}-u^{j}|_{-1}^{2}&\le& 2C_1(\alpha)(t^{n+m}-t^n)+\left(\hat\rho_0^\prime\right)^2,\quad n\ge n_0,~~m>0.\label{eq4.36}
\end{eqnarray}
To infer the inequality (\ref{eq4.36}), which will be used in the proof of the uniform estimate for the discrete chemical potential $\omega^{n+1}$, we have employed the relation $|u^{j}-u^{j-1}|_{-1}=|\overline {u^{j}}-\overline {u^{j-1}}|_{-1}$.

\subsection{Time uniform bound in $\mathcal H^1_\alpha$ }
We now establish the uniform a priori estimate for the numerical solution in $\mathcal H^1_\alpha$.

\begin{theorem}[Time uniform bound in $\mathcal H^1_\alpha$]\label{th2} Assume that (\ref{eqa2.4}), (\ref{eqa2.5}), (\ref{eqa2.6}), (\ref{eqa2.7}), and (\ref{eq2.18a}) are fulfilled. Let $u^{n+1}$ be the solution of the numerical scheme (\ref{eq3.4a}). Then for any $r>0$ and any time step sequence $\{k_n\}$ with $k=\sup_{n\ge 0}k_n$ satisfying (\ref{eqa4.12}), there exists a positive integer $n_1=n_1(R,\alpha,\hat\rho_{1},k,r)$ such that
\begin{eqnarray}\label{eq4.35}
|u^{n}|_1 \leq  \hat\rho_{1},\qquad n\ge n_1,
\end{eqnarray}
with
\begin{eqnarray}
\hat\rho_{1}=\left(\frac{2}{\varepsilon} C_{1}(\alpha )+\frac{\left(\hat\rho_{0}^\prime\right)^2}{r\varepsilon }\right)^{\frac{1}{2}}. \nonumber
\end{eqnarray}This leads to an absorbing set in the space $\mathcal H^1_\alpha$ of radius $\hat\rho_1$ for $S_{k}$ and implies the existence of a compact global attractor $\mathcal A_k\subset \mathcal H^1_\alpha$ for the scheme for each time grid $\mathcal J$ with $k\le \frac{8\varepsilon}{c^2}$ and $$\sup_{u\in \mathcal A}|u|_1\le \hat\rho_{1},$$
where $\mathcal A$ is the union of the global attractors for the scheme with different time grids $\mathcal J$.
\end{theorem}

\begin{proof} Adding up (\ref{eq4.21}) with $n$ from $n_0$ to $m$
\begin{eqnarray}
\frac{1}{2}\left(|\overline {u^{m+1}}|_{-1}^{2}-|\overline {u^{n_0}}|_{-1}^{2}+\sum\limits_{i=n_0}^m|\overline {u^{i+1}}-\overline {u^{i}}|_{-1}^{2}\right)&+&\sum\limits_{i=n_0}^mk_i\mathcal E(u^{i+1})\nonumber\\
&\leq& C_{0}(\alpha )\sum\limits_{i=n_0}^mk_i,
\end{eqnarray}
where $n_0$ has been defined in (\ref{eqa4.27}). The bound in $H^{-1}$ is essentially there already.
This implies that
\begin{eqnarray}
\sum\limits_{i=n_0}^mk_i\mathcal E(u^{i+1})\leq C_{0}(\alpha )\sum\limits_{i=n_0}^mk_i+\frac{\left(\hat\rho_{0}^\prime\right)^2}{2}
\end{eqnarray}
Since by Proposition \ref{pro4.1} $\mathcal E(u^{i+1})$ decays, we have
\begin{eqnarray}\label{eq4.31}
\mathcal E(u^{m+1})\leq C_{0}(\alpha )+\frac{\left(\hat\rho_{0}^\prime\right)^2}{2(t^{m+1}-t^{n_0})}.
\end{eqnarray}
Noticing that the assumption (\ref{eqa2.6}) implies
\begin{eqnarray*}
\int_{\Omega}F( u^{m+1})dx \geq \int_{\Omega}\left[\frac{1}{2}c_0( u^{m+1})^{2p}-c_{3}\right]dx,
\end{eqnarray*}
we conclude from (\ref{eq4.31}) that
\begin{eqnarray*}
\int_{\Omega}\frac{1}{2}c_0( u^{m+1})^{2p}dx+\frac{\varepsilon }{2}| u^{m+1}|_1^2\leq C_{1}(\alpha )+\frac{\left(\hat\rho_{0}^\prime\right)^2}{2(t^{m+1}-t^{n_0})},
\end{eqnarray*}
which implies
\begin{eqnarray}
|u^{m+1}|_1^2\leq \frac{2}{\varepsilon} C_{1}(\alpha )+\frac{\left(\hat\rho_{0}^\prime\right)^2}{\varepsilon(t^{m+1}-t^{n_0})}.
\end{eqnarray}
Then for any $r>0$ (say $r=1$), there exists a positive integer $n_1$ such that $t^{n_1}-t^{n_0}>r$ and hence
when $n\ge n_1$
\begin{eqnarray}
| u^{n}|_1^2\leq \frac{2}{\varepsilon} C_{1}(\alpha )+\frac{\left(\hat\rho_{0}^\prime\right)^2}{r\varepsilon }.
\end{eqnarray}
This completes the proof of Theorem \ref{th2}.
\qquad\end{proof}

Observe that the uniform bound of the numerical solutions obtained in Theorem \ref{th2} is similar to that of the exact solution. We would like to emphasize the point that Theorem \ref{th2} implies that trajectories starting from $B_R(\mathcal H_\alpha)$ (ball centered at the origin with radius $R$ in the affine space $\mathcal H_\alpha$ with the norm $|\cdot|_{-1}$) enter an absorbing ball in $\mathcal H^1_\alpha$ of radius $\hat\rho_1$ with approximately $n_1$ steps.

From (\ref{eq4.31}), we deduce that
\begin{eqnarray}\label{eqaa4.43}
\mathcal E(u^{m+1})\leq C_{0}(\alpha )+\frac{\left(\hat\rho_{0}^\prime\right)^2}{2r}:=C_2(r),\quad \forall r>0,\quad t^{m+1}-t^{n_0}\ge r,
\end{eqnarray}
which will be used in following analysis.

\begin{remark}[Alternative proof of uniform estimate in $\mathcal H^1_\alpha$] Adding up (\ref{eq4.21}) with $n$ from $0$ to $m$
\begin{eqnarray}\label{eq4.40}
&&\frac{1}{2}\left(|\overline {u^{m+1}}|_{-1}^{2}-|\overline {u}_{0}|_{-1}^{2}+\sum\limits_{i=0}^m|\overline {u^{i+1}}-\overline {u^{i}}|_{-1}^{2}\right)+\sum\limits_{i=0}^mk_i\mathcal E(u^{i+1})\nonumber\\
&\leq& C_{0}(\alpha )\sum\limits_{i=0}^mk_i.
\end{eqnarray}
Utilising the decay property of $\mathcal E(u^{i+1})$ and the assumption (\ref{eqa2.6}) yields
\begin{eqnarray}
|u^{m+1}|_1^2 \leq \frac{2}{\varepsilon}C_1(\alpha)+\frac{1}{\varepsilon t^{m+1}}|\overline {u}(0)|_{-1}^{2}.
\end{eqnarray}
Then when $|\overline {u}(0)|_{-1}\le R$, for any given $\rho>\tilde \rho_1=\left(\frac{2}{\varepsilon}C_1(\alpha)\right)^{\frac{1}{2}}$, there exists a positive integer $m$ such that $t^{m+1}>\frac{R}{\varepsilon(\rho^2-\tilde \rho^2_1)}$ and hence $| u^{m+1}|_1\le \rho.$
\end{remark}

\subsection{Time uniform bound in $\mathcal H^2_\alpha$} With the $\mathcal H^1_\alpha$ uniform bound established in the previous subsection, we now prove the $\mathcal H^2_\alpha$ uniform bound of the scheme (\ref{eq3.4a}) and state the following theorem.

\begin{theorem}[Time uniform bound in $\mathcal H^2_\alpha$]\label{th4.6} Assume that $|u_0|_2\le R_2$, and (\ref{eqa2.4}), (\ref{eqa2.5}), (\ref{eqa2.6}), (\ref{eqa2.7}), (\ref{eq2.18a}) are fulfilled. Let $u^{n+1}$ be the solution of the numerical scheme (\ref{eq3.4a}). Then for any time step sequence $\{k_n\}$ with $k=\sup_{n\ge 0}k_n$ satisfying (\ref{eqa4.12}), we have
\begin{eqnarray}\label{eq4.42}
|u^{n+1}|_2^2\le R_2^2+\frac{\gamma^\prime_2 t^{n+1}}{\varepsilon}=:\hat E_2(R_2,t^{n+1}),\qquad \forall n\ge 0,
\end{eqnarray}
where $\gamma^\prime_2$ will be given in (\ref{eq4.53}). Furthermore, for any given $\rho>\hat\rho_{2}$ with
\begin{equation}
\hat\rho_{2}=\left(\frac{\gamma^3_4\varepsilon }{3\sqrt{3}}\hat\rho_{1}^{2}+\frac{\gamma_3}{\varepsilon}\right)^{\frac{1}{2}}\left(\frac{8\varepsilon}{c^2}+\frac{2}{\varepsilon } \right)^{\frac{1}{2}},
\end{equation}
where $\gamma_3$ and $\gamma_4$ will be given in (\ref{eq4.44}) and (\ref{eq4.48}), respectively, there exists a positive integer $n_2=n_2(R_2,\alpha,\hat\rho_{2},k)$ such that
\begin{eqnarray}\label{eqa4.43}
|u^{n}|_2 \leq  \rho,\qquad n\ge n_2.
\end{eqnarray}
\end{theorem}

\begin{proof}
We multiply (\ref{eq3.4a}) by $\Delta^2 u^{n+1}$, integrate by parts using the Green formula and the boundary conditions. Then we obtain
\begin{eqnarray}
&&\frac{1}{2}\left(|u^{n+1}|_2^2-|u^{n}|_2^2+|u^{n+1}-u^n|_2^2\right)+\varepsilon k_n \|\Delta^2 u^{n+1}\|^2\nonumber\\
&=&k_n\left(\Delta f(u^{n+1}),\Delta^2 u^{n+1}\right)\nonumber\\
&\le&\frac{k_n}{2\varepsilon}|f(u^{n+1})|_2^2+\frac{\varepsilon k_n}{2}\|\Delta^2 u^{n+1}\|^2,
\end{eqnarray}
which implies that
\begin{eqnarray}\label{eq4.43}
|u^{n+1}|_2^2-|u^{n}|_2^2+|u^{n+1}-u^n|_2^2+\varepsilon k_n \|\Delta^2 u^{n+1}\|^2\le\frac{k_n}{\varepsilon}|f(u^{n+1})|_2^2.
\end{eqnarray}
To prove (\ref{eq4.42}) and (\ref{eqa4.43}), we need the following inequality (see, for example, \cite{Temam88,Sell02})
\begin{eqnarray}\label{eqa4.44}
|f(u^{n+1})|_2^2\le \gamma_2\left(1+\|\Delta^2 u^{n+1}\|^{2\sigma}\right),
\end{eqnarray}
where $0\le \sigma<1$, the condition $p=2$ when $d=3$ being precisely necessary to ensure that $\sigma<1$ in this case.

We first show (\ref{eq4.42}), that is, obtain an a priori bound of $|u^{i}|_2$, $i\ge 1$. From (\ref{eqa4.44}), one gets\begin{eqnarray}\label{eq4.53}
| f(u^{n+1})|_2^2\le \varepsilon^2\|\Delta^2 u^{n+1}\|^{2}+\gamma^\prime_2.
\end{eqnarray}Substitute (\ref{eq4.53}) into (\ref{eq4.43}) to obtain
\begin{equation}\label{eq4.54}
|u^{n+1}|_2^2-|u^{n}|_2^2+|u^{n+1}-u^n|_2^2\le \frac{\gamma^\prime_2 k_n}{\varepsilon}.
\end{equation}It follows from (\ref{eq4.54}) that
\begin{eqnarray}\label{eq4.50}
|u^{n+1}|_2^2\le |u^{n}|_2^2+\frac{\gamma^\prime_2 k_n}{\varepsilon}\le|u^{0}|_2^2+\frac{\gamma^\prime_2 t^{n+1}}{\varepsilon}\le\hat E_2(R_2,t^{n+1}),\qquad \forall n\ge 0.
\end{eqnarray}

We now show the uniform bound (\ref{eqa4.43}). Using the Young inequality, from (\ref{eqa4.44}), we have
\begin{eqnarray}\label{eq4.44}
|f(u^{n+1})|_2^2\le \frac{\varepsilon^2}{2}\|\Delta^2 u^{n+1}\|^{2}+\gamma_3.
\end{eqnarray}
Substitute (\ref{eq4.44}) into (\ref{eq4.43}) to obtain
\begin{equation}\label{eq4.41}
|u^{n+1}|_2^2-|u^{n}|_2^2+|u^{n+1}-u^n|_2^2+\frac{\varepsilon k_n}{2} \|\Delta^2 u^{n+1}\|^2\le \frac{\gamma_3 k_n}{\varepsilon}.
\end{equation}
By the Nirenberg-Gagliardo inequality, there exists a positive constant $\gamma_4$ such that
\begin{eqnarray}\label{eq4.48}
|u|_2\le \gamma_4|u|_1^{2/3}\|\Delta^2 u\|^{1/3}.
\end{eqnarray}
Using the Young inequality, we get
\begin{eqnarray*}
|u^{n+1}|_2^2\le \frac{2\gamma^3_4}{3\sqrt{3}}|u^{n+1}|_1^{2}+\|\Delta^2 u^{n+1}\|^{2},\quad \forall n\ge 0.
\end{eqnarray*}
which implies
\begin{eqnarray}
\frac{\varepsilon k_n}{2}\|\Delta^2 u^{n+1}\|^{2}\ge \frac{\varepsilon k_n}{2}|u^{n+1}|_2^2- \frac{\gamma^3_4\varepsilon k_n}{3\sqrt{3}}|u^{n+1}|_2^{2}, \quad \forall n\ge 0.
\end{eqnarray}
Substituting this inequality into (\ref{eq4.41}) yields
\begin{equation}\label{eq4.49}
\left(1+\frac{\varepsilon k_n}{2}\right)|u^{n+1}|_2^2\le |u^{n}|_2^2 + \frac{\gamma^3_4\varepsilon k_n}{3\sqrt{3}}|u^{n+1}|_1^{2}+\frac{\gamma_3 k_n}{\varepsilon}, \quad \forall n\ge 0.
\end{equation}

By induction, noticing the time uniform bound of $|u^{i}|_1$, $i\ge n_1$, we have from (\ref{eq4.49})
\begin{eqnarray}
|u^{n+1}|_2^2&\le& \prod\limits_{i=n_1}^n\frac{2}{2+\varepsilon k_i}|u^{n_1}|_2^2 + \sum\limits_{i=n_1}^n\prod\limits_{j=i}^n \left(\frac{2}{2+\varepsilon k_j}\right)\left(\frac{\gamma^3_4\varepsilon k_i}{3\sqrt{3}}|u^{i+1}|_1^{2}+\frac{\gamma_3 k_i}{\varepsilon}\right)\nonumber\\
&\le&\prod\limits_{i=n_1}^n\exp\left(-\frac{\varepsilon k_i}{2+\varepsilon k}\right)|u^{n_1}|_2^2 \nonumber\\
&&+ \sum\limits_{i=n_1}^n \prod\limits_{j=i}^n\exp\left(-\frac{\varepsilon k_j}{2+\varepsilon k}\right)\left(\frac{\gamma^3_4\varepsilon }{3\sqrt{3}}\hat\rho_{1}^{2}+\frac{\gamma_3}{\varepsilon}\right)k_i\nonumber\\
&\le&\exp\left(-\frac{\varepsilon (t^{n+1}-t^{n_1})}{2+\varepsilon k}\right)|u^{n_1}|_2^2 \nonumber\\
&&+ \left(\frac{\gamma^3_4\varepsilon }{3\sqrt{3}}\hat\rho_{1}^{2}+\frac{\gamma_3}{\varepsilon}\right)\left(k+\frac{2}{\varepsilon }\right)\left[1-\exp\left(-\frac{\varepsilon (t^{n+1}-t^{n_1})}{2+\varepsilon k}\right)\right].
\end{eqnarray}
Let there be given $\rho>\hat \rho_{2}$. Then there exists a positive integer $n_2$ with
\begin{equation}
t^{n_2}\ge t^{n_1}+\frac{(2+\varepsilon k)[\ln \hat E_2(R_2,t_{n_1+1})-\ln(\rho^2-\hat\rho_{2}^2)]}{\varepsilon}
\end{equation}
such that (\ref{eqa4.43}) holds. The proof thus is complete.
\end{proof}

Theorem \ref{th4.6} implies that for any $\epsilon>0$ there exists constant $\hat\rho_2+\epsilon$ independent of the time step-sizes $k_n$ such the discrete semigroup $S_k$ possesses an absorbing set in $\mathcal H^2_\alpha$ with radius $\hat\rho_2+\epsilon$ which attracts all bounded sets in $\mathcal H^2_\alpha$. In particular, this scheme is uniformly dissipative with
\begin{eqnarray}
\sup\limits_{u\in K}|u|_2\le \hat\rho_2+\epsilon,
\end{eqnarray}
where $K$ is the union of the global attractors for the scheme with different time meshes $\mathcal J$.

There are two byproducts of the inequality (\ref{eq4.41}). Namely, we have
\begin{eqnarray}
\frac{\varepsilon}{2}\sum\limits_{i=0}^nk_n\|\Delta^2 u^{i+1}\|^{2}&\leq& R_2^2+\frac{\gamma_3t^{n+1}}{\varepsilon},\quad \forall n\ge 0.\\
\sum\limits_{i=0}^n|u^{i+1}- u^{i}|_2^{2}&\leq& R_2^2+\frac{\gamma_3t^{n+1}}{\varepsilon},\quad \forall n\ge 0.\end{eqnarray}
Since $\|\Delta^2\cdot\|$ is equivalent to $\|\cdot\|_4$, the first inequality is a bound in $L^2(0,T^*;\dot H^4)$ for the scheme for any $T^*>0$. It is easy to see that the second is a bound on the difference of the solution at adjacent time steps.

For the Neumann boundary condition, we have a much simpler proof for the uniform $\mathcal H^2_\alpha$ estimate.

\begin{remark}[Alterative proof of the uniform bound in $\mathcal H^2_\alpha$ for the Neumann boundary condition] In the case of the Neumann boundary condition, there exists a positive constant $\gamma_5$ such that for any $|u_0|_2\le R_2$ (see, Lemma 55.5 in \cite{Sell02})
\begin{eqnarray}
|u^{n+1}|_2^2\le \gamma_5\|\Delta^2 u^{n+1}\|^{2},\qquad n\ge 0.
\end{eqnarray}
Substituting this inequality into (\ref{eq4.41}) yields
\begin{equation}
\left(1+\frac{\varepsilon \gamma_5k_n}{2}\right)|u^{n+1}|_2^2\le |u^{n}|_2^2 +\frac{\gamma_3 k_n}{\varepsilon},
\end{equation}
which implies
\begin{eqnarray}
|u^{n+1}|_2^2\le \frac{2}{2+\varepsilon \gamma_5k_n}|u^{n}|_2^2 +\frac{2\gamma_3 k_n}{\varepsilon(2+\varepsilon \gamma_5k_n)}.
\end{eqnarray}
By induction, we further have
\begin{eqnarray}\label{eq4.64}
|u^{n+1}|_2^2&\le& \prod\limits_{i=0}^n\frac{2}{2+\varepsilon \gamma_5 k_i}| u^{0}|_2^2 + \sum\limits_{i=0}^n\prod\limits_{j=i}^n \left(\frac{2}{2+\varepsilon \gamma_5k_j}\right)\frac{\gamma_3 k_i}{\varepsilon}\nonumber\\
&\le& \prod\limits_{i=0}^n\left(1-\frac{\varepsilon \gamma_5k_i}{2+\varepsilon \gamma_5k}\right)| u^{0}|_2^2 +\sum\limits_{i=0}^n\prod\limits_{j=i}^n \left(1-\frac{\varepsilon \gamma_5k_j}{2+\varepsilon \gamma_5k}\right)\frac{\gamma_3 k_i}{\varepsilon}\nonumber\\
&\le& \exp\left(-\frac{\varepsilon \gamma_5t^{n+1}}{2+\varepsilon \gamma_5k}\right)|u^{0}|_2^2 +\frac{\gamma_3(2+\varepsilon \gamma_5k)}{\varepsilon^2 \gamma_5}.
\end{eqnarray}
Define
\begin{equation}
\tilde\rho_{2k}=\frac{1}{\varepsilon}\left(\frac{\gamma_3(2+\varepsilon \gamma_5k)}{ \gamma_5} \right)^{\frac{1}{2}},
\end{equation}
and let there be given $\rho>\tilde\rho_{2k}$. Then there exists a positive integer $\tilde n_2$ with
\begin{equation*}
t^{\tilde n_2}\ge \frac{(2+\varepsilon \gamma_5k)[\ln R_2-\ln(\rho^2-\tilde\rho_{2k}^2)]}{\varepsilon \gamma_5}
\end{equation*}
such that
\begin{equation}\label{eq4.66}|u^{n}|_2\le \rho,\qquad \forall n\ge \tilde n_2.\end{equation}

We want to emphasize that (\ref{eq4.66}) holds for any time step sequence $\{k_n\}$. This implies that it is an unconditional $\mathcal H^2_\alpha$ estimate. When (\ref{eqa4.12}) is satisfied, we have the following uniform estimate independent of the initial data and the time step-sizes $k_n$: there exists a positive integer $\hat n_2$ such that
\begin{equation}\label{eq4.67}|u^{n}|_2\le \rho,\qquad \forall n\ge \hat n_2,\end{equation}
for any $\rho>\tilde\rho_2$ with $\tilde\rho_{2}=\frac{1}{\varepsilon}\left(\frac{2\gamma_3}{\gamma_5}+\frac{8\gamma_3\varepsilon^2}{ c^2} \right)^{\frac{1}{2}}$.
\end{remark}

\subsection{Time uniform bound for chemical potential in $H^1$}
Our goal now is to obtain uniform bound in $H^1$ for the discrete chemical potential $\omega^n$ of the scheme (\ref{eq3.4a}). This estimate will be used to show the uniform $\mathcal H^3_\alpha$ estimate. For this
purpose, we need the following discrete uniform Gronwall lemma with variable step-sizes, slightly different from Appendix 2 in \cite{Elliott93} and Lemma 4.6 in \cite{Tone12}.

\begin{lemma}[Discrete Uniform Gronwall Lemma]\label{lem2} Let $n_*$, $N\in \mathbb N$, $a_1$, $a_2$, $a_3$, $\hat r >0$, and $\{k_n\}$, $\{\xi^n\}$, $\{\zeta^n\}$, $\{\eta^n\}$ be four sequences of nonnegative real numbers which satisfy
\begin{eqnarray}k\eta^n&\le&\vartheta\in (0,1), \quad n\ge n_*,\qquad k=\sup\limits_{n\ge n_*}k_n,\\
(1-k_n\eta^n)\xi^{n}&\le& \xi^{n-1}+k_n\zeta^n,\quad n\ge n_*,\label{eq4.71a}
\end{eqnarray}
and
\begin{eqnarray}\label{eq4.72a}\sum\limits_{n=n^*}^{n^*+N}k_n\eta^n\le a_1,\quad \sum\limits_{n=n^*}^{n^*+N}k_n\zeta^n\le a_2,\quad \sum\limits_{n=n^*}^{n^*+N}k_n\xi^n\le a_3,
\end{eqnarray}
for all $n^*\ge n_*$, with $\hat r=t^{n_*+N}-t^{n_*}>0$. Then
$$\xi^n\le \left(a_2+\frac{a_3}{\hat r}\right)\exp(c_\vartheta a_1),\qquad \forall n\ge n_*+N,$$
where the constant $c_\vartheta>0$ depends on $\vartheta$.
\end{lemma}
\begin{proof} The proof is similar with that of Appendix 2 in \cite{Elliott93} and Lemma 4.6 in \cite{Tone12}. Let $n_{*1}$ and $n_{*2}$ be such that $n_*\le n_{*1}<n_{*2}\le n_{*1}+N$. We use recursively (\ref{eq4.71a}) to obtain
\begin{eqnarray*}
\xi^{n_{*1}+N}\le\frac{1}{\prod_{n=n_{*2}}^{n_{*1}+N}(1-k_n\eta^n)}\xi^{n_{*2}-1}
+\sum\limits_{n=n_{*2}}^{n_{*1}+N}\frac{k_n}{\prod_{j=n}^{n_{*1}+N}(1-k_j\eta^j)}\zeta^n.
\end{eqnarray*}
Since for every $\vartheta_n\in (0,1)$, there exists a constant $c_{\vartheta_n}$ such that $\frac{1}{1-\vartheta_n}\le e^{c_{\vartheta_n} \vartheta_n}\le e^{c_{\vartheta} \vartheta_n}$ with $c_\vartheta=\max_{n=n_{*2},\cdots,n_{*1}+N}\{c_{\vartheta_n}\}$, from the assumption (\ref{eq4.72a}), we obtain
\begin{eqnarray*}
\xi^{n_{*1}+N}\le\left(\xi^{n_{*2}-1}+a_2\right)e^{c_\vartheta a_1}.
\end{eqnarray*}
Multiplying this inequality by $k_{n_{*2}-1}$, summing $n_{*2}$ from $n_{*1}+1$ to $n_{*1}+N$ and using the third inequality in (\ref{eq4.72a}), we get the desired result.
\end{proof}

The following lemma will prove useful.
\begin{lemma}Assume that (\ref{eqa2.7}) holds. Let $\omega^{n+1}$ be the discrete chemical potential obtained by the numerical scheme (\ref{eq3.4})-(\ref{eq3.5}). Then
\begin{eqnarray}\label{eq4.71}
\bar\partial_t|\omega^{n+1}|_1^2+k_n|\bar\partial_t \omega^{n+1}|_1^2\le \frac{c^2}{2\varepsilon}|\omega^{n+1}|_1^2,\quad n\ge0.
\end{eqnarray}
\end{lemma}
\begin{proof}We take the inner product of (\ref{eq3.4}) with $\bar\partial_t \omega^{n+1}$ and obtain
\begin{eqnarray}\label{eq4.72}
(\bar\partial_t u^{n+1},\bar\partial_t \omega^{n+1})+\frac{1}{2}\left[\bar\partial_t|\omega^{n+1}|_1^2+k_n|\bar\partial_t \omega^{n+1}|_1^2\right]=0.
\end{eqnarray}
We note that by (\ref{eqa2.7}) and the mean value theorem we have, for all $v_1,v_2\in \mathbb R$,
\begin{eqnarray}
(f(v_1)-f(v_2))(v_1-v_2)=f^\prime(\xi_{v_1,v_2})(v_1-v_2)^2\ge -c(v_1-v_2)^2, \quad \xi_{v_1,v_2}\in \mathbb R.~~~
\end{eqnarray}
It follows that
\begin{eqnarray}\label{eq4.74}
(\bar\partial_t u^{n+1},\bar\partial_t \omega^{n+1})&=&\varepsilon|\bar\partial_t u^{n+1}|_1^2+(\bar\partial_t u^{n+1},\bar\partial_t f(u^{n+1}))\nonumber\\
&\ge&\varepsilon|\bar\partial_t u^{n+1}|_1^2-c\|\bar\partial_t u^{n+1}\|^2.
\end{eqnarray}
Substituting (\ref{eq4.74}) into (\ref{eq4.72}) yields
\begin{eqnarray}\label{eq4.75}
\frac{1}{2}\left[\bar\partial_t|\omega^{n+1}|_1^2+k_n|\bar\partial_t \omega^{n+1}|_1^2\right]+\varepsilon|\bar\partial_t u^{n+1}|_1^2&\le& c\|\bar\partial_t u^{n+1}\|^2\nonumber\\
&\le&\frac{c^2}{4\varepsilon}|\bar\partial_t u^{n+1}|^2_{-1}+\varepsilon\|\bar\partial_t u^{n+1}\|_1^2.~~~
\end{eqnarray}
From (\ref{eq3.4}) we also have
$$|\bar\partial_t u^{n+1}|_{-1}=|\omega^{n+1}|_1.$$
Then (\ref{eq4.71}) follows from (\ref{eq4.75}).
\end{proof}

\begin{theorem}[Time uniform bound of chemical potential in $H^1$]\label{th5} Assume that (\ref{eqa2.4}), (\ref{eqa2.5}), (\ref{eqa2.6}), (\ref{eqa2.7}), and (\ref{eq2.18a}) are fulfilled. Let $\omega^{n+1}$ be the discrete chemical potential obtained by the numerical scheme (\ref{eq3.4})-(\ref{eq3.5}). Then for any time step sequence $\{k_n\}$ with $k=\sup_{n\ge 0}k_n$ satisfying
\begin{eqnarray}\label{eq4.78}
k<\frac{2\varepsilon}{c^2},
\end{eqnarray}
there exists a positive integer $n_\omega$ with $r_1=t^{n_\omega}-t^{n_0+N_1}>0$ and $t^{n_0+N_1}-t^{n_0}\ge r$ such that
\begin{eqnarray}\label{eqa4.75}
|\omega^{n}|_1 \leq  \rho_\omega,\qquad n\ge n_\omega,
\end{eqnarray}
where
\begin{eqnarray}
\rho_\omega^2:=\left(\frac{c^2}{4\varepsilon}C_1(\alpha)+\frac{c^2\left(\hat\rho_0^\prime\right)^2}{8\varepsilon r_1}+\frac{C_2(r)}{r_1}\right)\exp\left(\frac{2c^2r_1}{\varepsilon}\right).
\end{eqnarray}
\end{theorem}

\begin{proof} We take the inner product of (\ref{eq3.4}) with $\omega^{n+1}$ and obtain
\begin{eqnarray}
(\bar\partial_tu^{n+1},\omega^{n+1})+|\omega^{n+1}|_1^2=0,
\end{eqnarray}
which implies that
\begin{eqnarray}
\frac{\varepsilon}{2k_n}\left[|u^{n+1}|_1^2-|u^{n}|_1^2+|u^{n+1}-u^n|_1^2\right]+(\bar\partial_tu^{n+1},f(u^{n+1}))+|\omega^{n+1}|_1^2=0.
\end{eqnarray}
Using (\ref{eqa4.15}), we further have
\begin{eqnarray}
\frac{\varepsilon k_n}{2}|\bar\partial_tu^{n+1}|_1^2+\bar\partial_t\mathcal E(u^{n+1})+|\omega^{n+1}|_1^2\le \frac{ck_n}{2}\|\bar\partial_t u^{n+1}\|^2,
\end{eqnarray}
and therefore
\begin{eqnarray}\label{eq4.83}
\bar\partial_t\mathcal E(u^{n+1})+|\omega^{n+1}|_1^2\le \frac{c^2k_n}{8\varepsilon}|\bar\partial_t u^{n+1}|_{-1}^2.
\end{eqnarray}
Summing (\ref{eq4.83}) from $n=n_0+N_1$ to $n_0+N_1+N_2$ with $N_1,N_2\in \mathbb N$ and $t^{n_0+N_1}-t^{n_0}\ge r$, we obtain
\begin{eqnarray}
&&\mathcal E(u^{n_0+N_1+N_2+1})-\mathcal E(u^{n_0+N_1})+\sum\limits_{i=n_0+N_1}^{n_0+N_1+N_2}k_{i}|\omega^{i+1}|_1^2\nonumber\\
&\le&
\frac{c^2}{8\varepsilon}\sum\limits_{i=n_0+N_1}^{n_0+N_1+N_2}|u^{i+1}-u^i|_{-1}^2.
\end{eqnarray}
Using (\ref{eq4.36}) and (\ref{eqaa4.43}), we get
\begin{eqnarray}\label{eq4.85}
\sum\limits_{i=n_0+N_1}^{n_0+N_1+N_2}k_{i}|\omega^{i+1}|_1^2\le
\frac{c^2}{4\varepsilon}C_1(\alpha)(t^{n_0+N_1+N_2+1}-t^{n_0+N_1})+\frac{c^2}{8\varepsilon}\left(\hat\rho_0^\prime\right)^2+C_2(r),~~~
\end{eqnarray}
where $n_0$ has been defined in (\ref{eqa4.27}).

When $r_1=t^{n_0+N_1+N_2+1}-t^{n_0+N_1}>0$, an application of Lemma \ref{lem2} to (\ref{eq4.71}), in view of (\ref{eq4.85}), leads to (\ref{eqa4.75}). This completes the proof.
\end{proof}

\subsection{Time uniform bound in $\mathcal H^3_\alpha$}

We are now ready to give the uniform bound in $\mathcal H^3_\alpha$.

\begin{theorem}[Time uniform bound in $\mathcal H^3_\alpha$]\label{th5a} Assume that $|u_0|_2\le R_2$, and (\ref{eqa2.4}), (\ref{eqa2.5}), (\ref{eqa2.6}), (\ref{eqa2.7}), (\ref{eq2.18a}) are fulfilled. Let $u^{n+1}$ be the solution of the numerical scheme (\ref{eq3.4a}). Then for any time step sequence $\{k_n\}$ with $k=\sup_{n\ge 0}k_n$ satisfying (\ref{eq4.78}), we have the uniform estimates
\begin{eqnarray}\label{eqa4.83}
|u^{n}|_3 \leq  \hat\rho_3,\qquad n\ge n_3:=\max\{n_\omega,n_1\},
\end{eqnarray}
with $\hat\rho_3=\gamma_6\rho_\omega+\gamma_6\gamma_7\hat\rho_1$, where $\gamma_6$ and $\gamma_7$ will be defined in (\ref{eq4.84}) and (\ref{eq4.86}), respectively.
\end{theorem}
\begin{proof} Since, with constant $\gamma_6>0$,
\begin{eqnarray}\label{eq4.84}
|u^n|_3\le \gamma_6|\varepsilon \Delta u^n|_1,
\end{eqnarray}
we have
\begin{eqnarray}
|u^n|_3\le \gamma_6|\varepsilon \Delta u^n|_1\le \gamma_6|\omega^n|_1+\gamma_6|f(u^n)|_1.
\end{eqnarray}
In view of (see \cite{Li98})
\begin{eqnarray}\label{eq4.86}|f(u)|_1\le \gamma_7|u|_1,
\end{eqnarray} combining (\ref{eq4.35}) and (\ref{eqa4.75}) leads to (\ref{eqa4.83}).
\end{proof}

From uniform estimate (\ref{eqa4.83}), we know that the discrete semigroup $S_k$ possesses an absorbing set in $\mathcal H^3_\alpha$ with radius $\hat\rho_3$ which attracts all bounded sets in $\mathcal H^2_\alpha$.

\section{Numerical results} Using several numerical examples, we now illustrate the long time behaviour of the numerical solution. To do this, we first need to consider the space discretization of Cahn-Hilliard equation. Here we consider the Fourier pseudospectral approximation.
\subsection{Fourier spectral collocation methods}  For simplicity of presentation we consider a $2$-D domain $\Omega=(0,L_x)\times (0,L_y)$ and assume that $L_x=L_y=2\pi$ and $L_x=M_x\cdot h_x$, $L_y=M_y\cdot h_y$ for some mesh sizes $h_x=h_y=h>0$ and some positive integers $M_x=M_y=2M+1$. All variables are evaluated at the regular numerical grid $(x_i,y_j)$ with $x_i=ih,~y_j=jh$, $0\le i,j\le 2M+1$.

For a periodic function $g$ over the given $2$-D numerical grid, assume its discrete Fourier expansion is given by
\begin{eqnarray}
g_{i,j}=\sum\limits_{l,m=-M}^M\hat g^c_{l,m}e^{ i(lx_i+my_j)},
\end{eqnarray}
in which the collocation coefficients are given by the following backward transformation formula:
\begin{eqnarray}
\hat g_{l,m}=\frac{1}{(2M+1)^2}\sum\limits_{i,j=0}^{2M}f_{i,j}e^{- i(lx_i+my_j)}.
\end{eqnarray}
In the following numerical experiments, we need to calculate the discrete gradient, Laplacian, divergence, and $\nabla\Delta$ which become
\begin{eqnarray}
\nabla_Mg&=&\begin{pmatrix}
\mathcal D_{M_x}g\\
\mathcal D_{M_y}g
\end{pmatrix},\qquad \Delta_Mg=\nabla_M\cdot\nabla_Mg=\left(\mathcal D^2_{M_x}+\mathcal D^2_{M_y}\right)g,\\
\nabla_M\cdot\begin{pmatrix}g_1\\
g_2\end{pmatrix}&=&\mathcal D_{M_x}g_1+\mathcal D_{M_y}g_2,\qquad \nabla_M\Delta_Mg=\begin{pmatrix}\left(\mathcal D^3_{M_x}+\mathcal D_{M_x}\mathcal D^2_{M_y}\right)g\\
\left(\mathcal D_{M_y}\mathcal D^2_{M_x}+\mathcal D^3_{M_y}\right)g\end{pmatrix},
\end{eqnarray}
at the pointwise level. Here the Fourier spectral collocation approximation to the first, second and third order partial derivatives are given by
\begin{eqnarray}
\left(\mathcal D_{M_x}g\right)_{i,j}&=&\sum\limits_{l,m=-M}^M(l i)\hat g^c_{l,m}e^{ i(lx_i+my_j)},\\
\left(\mathcal D^2_{M_x}g\right)_{i,j}&=&\sum\limits_{l,m=-M}^M(-l^2)\hat g^c_{l,m}e^{ i(lx_i+my_j)}\\
\left(\mathcal D^3_{M_x}g\right)_{i,j}&=&\sum\limits_{l,m=-M}^M(-l^3i)\hat g^c_{l,m}e^{ i(lx_i+my_j)}.
\end{eqnarray}
The differentiation operators $\mathcal D_{M_y}$, $\mathcal D^2_{M_y}$ and $\mathcal D^3_{M_y}$ could be defined in the same fashion.

For any given periodic grid functions $\phi$ and $\psi$, the spectral approximations to the $L^2$ inner product and $L^2$ norm are introduced as
\begin{eqnarray}
\|\phi\|_{h}=\sqrt{( \phi,\phi)_h}\quad {\hbox{with}}\quad (\phi,\psi)_h=h^2\sum\limits_{i,j=0}^{2M}\phi_{i,j}\psi_{i,j}.
\end{eqnarray}
%The following formulas of summation by parts have been shown to be valid at the discrete level (see, e.g., []):
%\begin{eqnarray}
%\left\langle \phi,\nabla_M\cdot\begin{pmatrix}\psi_1\\
%\psi_2\end{pmatrix}\right\rangle&=&-\left\langle \nabla_M\phi,\begin{pmatrix}\psi_1\\
%\psi_2\end{pmatrix}\right\rangle,
%\qquad \left\langle \phi,\Delta_M\psi\right\rangle=-\left\langle \nabla_M\phi,\nabla_M\psi\right\rangle,\\
% \left\langle \phi,\Delta^2_M\psi\right\rangle&=&\left\langle \Delta_M\phi,\Delta_M\psi\right\rangle.
%\end{eqnarray}
\subsection{Numerical experiments} We consider a uniform space grid with $M_x=M_y=80$ and a uniform time grid $\mathcal J$ with $k=0.1$. Let $f(u)=u^3-u$ and $\varepsilon=0.1$. Obviously, the condition (\ref{eq4.78}) is satisfied. To support our uniform estimates which are independent of the initial data, we consider three different initial values:
\begin{enumerate}
\item[(i)] Initial value I: \begin{eqnarray}
u_0(x,y)=\left\{\begin{array}{ll} 1,&\qquad {\hbox{if}} ~~~~(x-\pi)^2+(y-\pi)^2<1,\\
-1,&\qquad {\hbox{others}};
\end{array}\right.
\end{eqnarray}
\item[(ii)] Initial value II ((2.29) in \cite{Xu18}):
\begin{eqnarray}u_0(x,y)=\tanh\left(\frac{\sqrt{x^2+y^2}-0.17}{\sqrt{2}\varepsilon}\right);
\end{eqnarray}
\item[(iii)] initial value III: \begin{eqnarray}u_0(x,y)=\sin(4x)\cos(3y).
\end{eqnarray}
\end{enumerate}
We perform the computation up to $T=100$ and examine the time evolution of the solution by recording its discrete $H^s$ ($s=1,2,3$) norms, namely, $\|\nabla_M u\|_h$,  $\|\Delta_M u\|_h$, and $\|\nabla_M\Delta u\|_h$, respectively. The discrete $H^1$ norm $\|\nabla_M \omega\|_h$ of the chemical potential $\omega$ is also computed. These time evolution plots are presented in Fig. \ref{fig:1}. A clear observation shows that, for all three initial data, all three norms of the solution $u$ and the discrete $H^1$ norm $\|\nabla_M \omega\|_h$ of the chemical potential $\omega$ are globally in time bounded. Also, these global in time bounds are expected to be valid for even a larger time scale. This numerical result provides strong evidence of the long time uniform estimates given in this paper.
\begin{figure}
\centering
\subfigure[The discrete $H^1$ norm of $u$.]{
\label{fig:a1} %% label for first subfigure
\includegraphics[height=5cm,width=5.5cm]{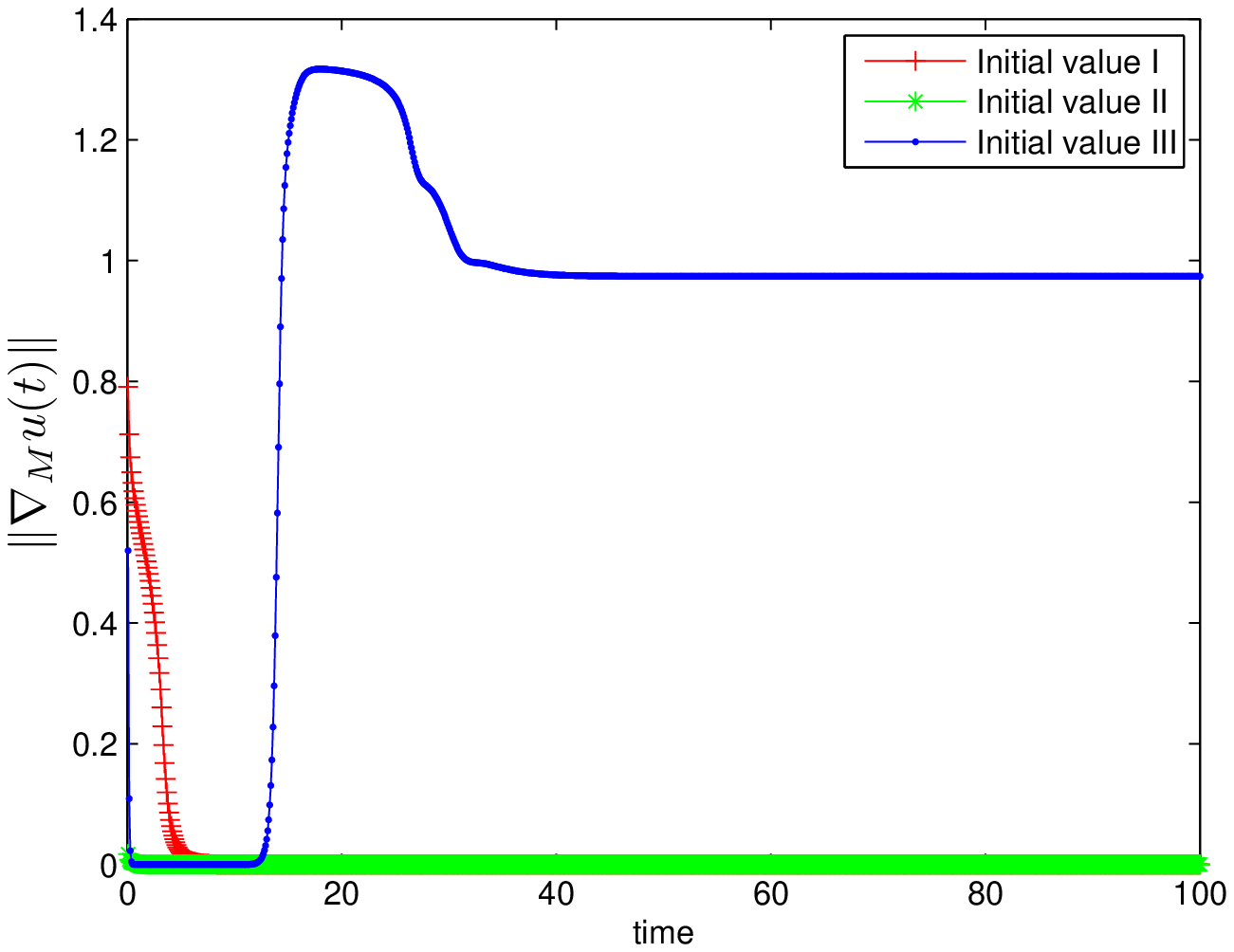}}
\subfigure[The discrete $H^2$ norm of $u$.]{
\label{fig:b1} %% label for second subfigure
\includegraphics[height=5cm,width=5.5cm]{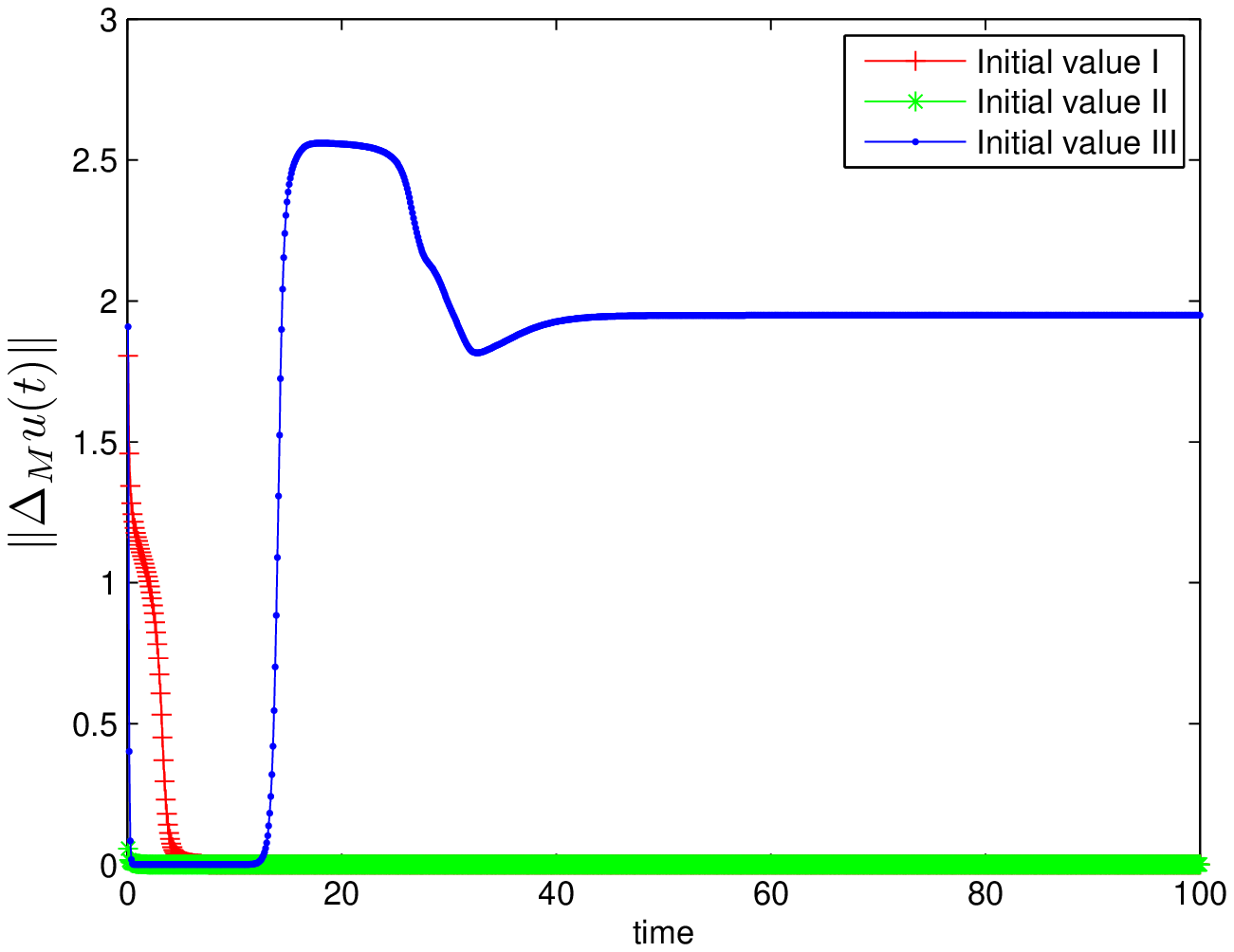}}
\subfigure[The discrete $H^3$ norm of $u$.]{
\label{fig:c1} %% label for first subfigure
\includegraphics[height=5cm,width=5.5cm]{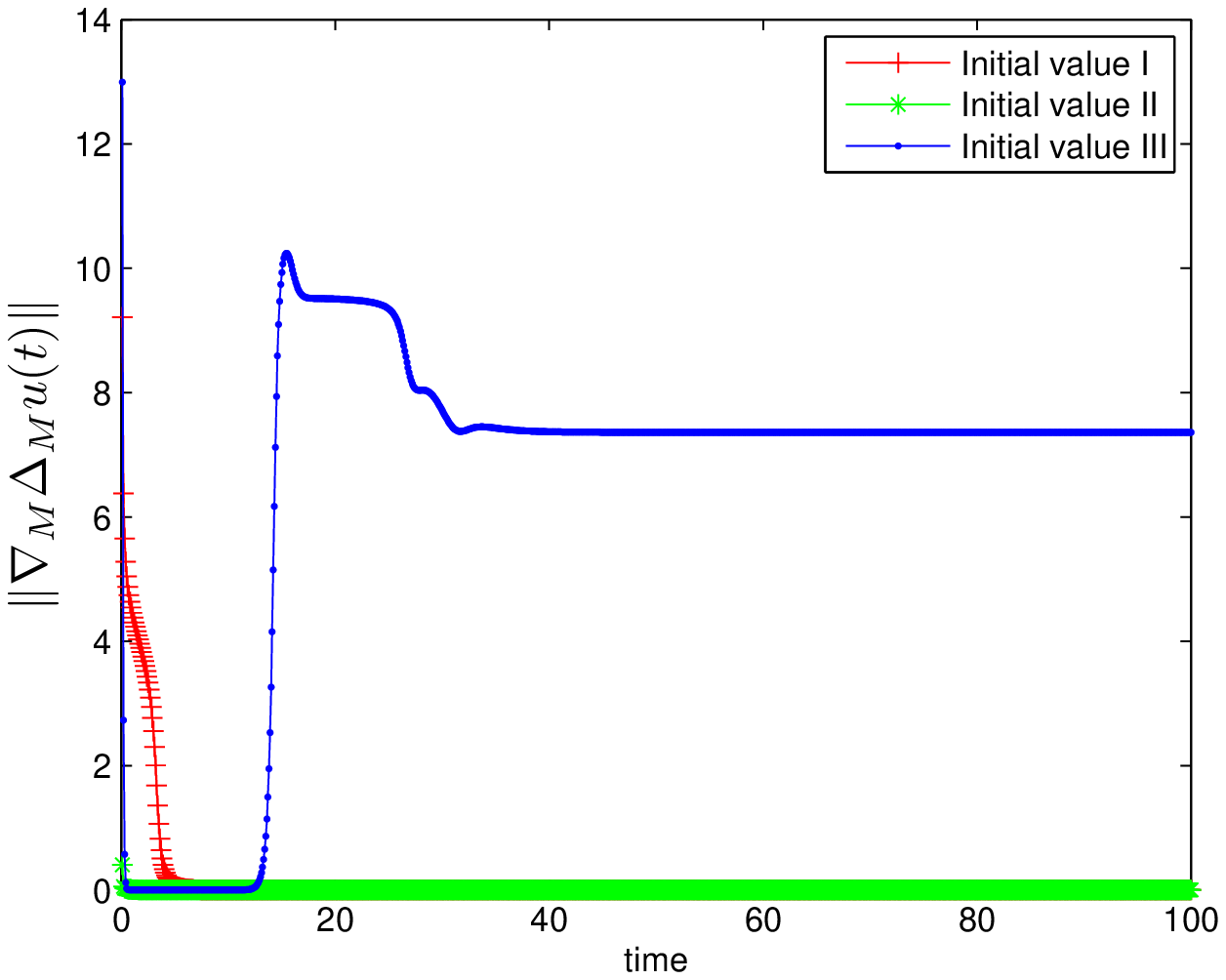}}
\subfigure[The discrete $H^1$ norm of $\omega$. ]{
\label{fig:d1} %% label for second subfigure
\includegraphics[height=5cm,width=5.5cm]{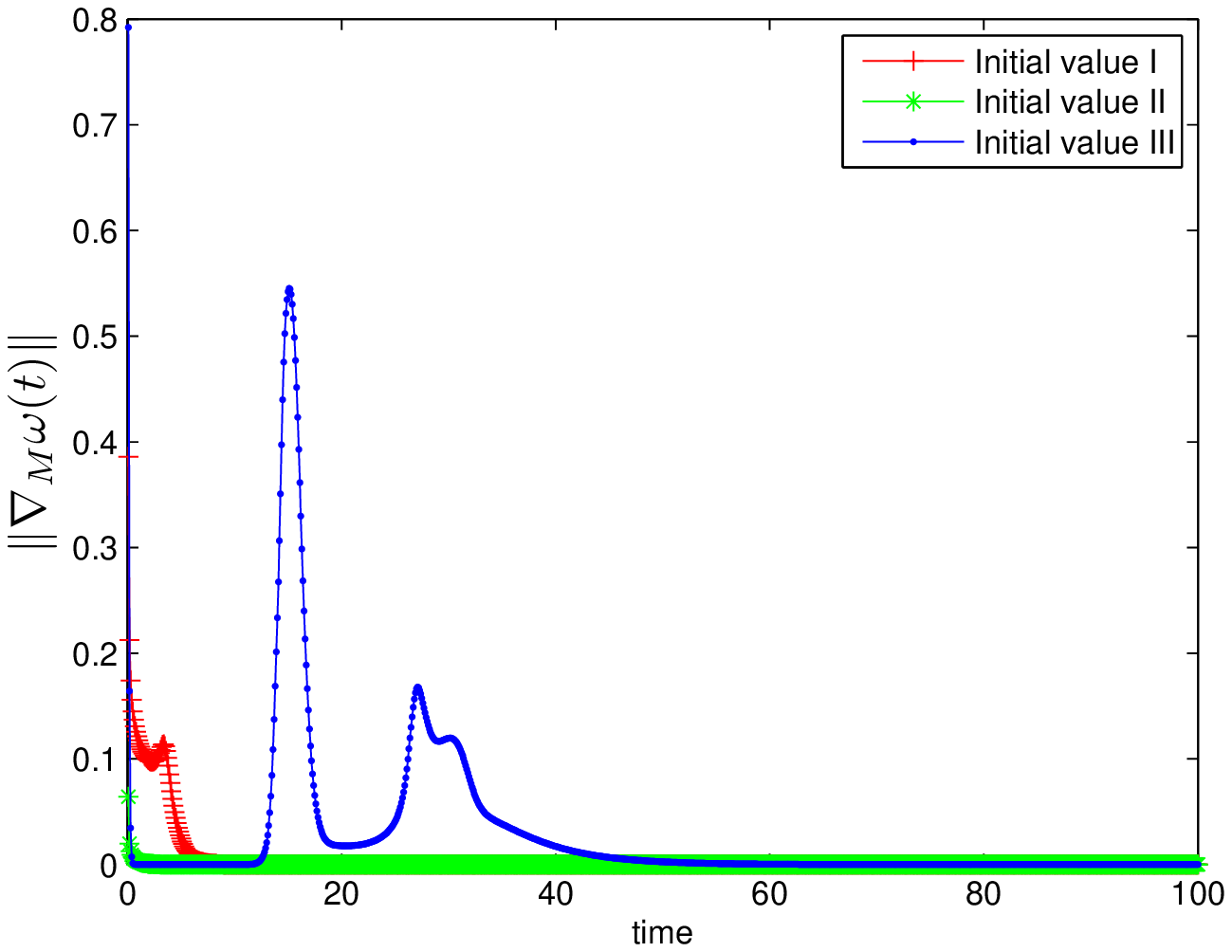}}
\caption{Time evolution of the solution $u$ and the chemical potential $\omega$, where $k=0.1,~\varepsilon=0.1$. }
\label{fig:1} %% label for entire figure
\centering
\end{figure}

Fig. \ref{fig:2} shows the different initial values and the evolutions of the numerical solutions
at different $t$'s, using the implicit Euler scheme for time discretization. It can be observed that the solutions
in all these cases eventually tend to be steady.
\begin{figure}
\centering
\subfigure[Initial value I, $t=0$.]{
\label{fig2:a1} %% label for first subfigure
\includegraphics[height=3cm,width=3.5cm]{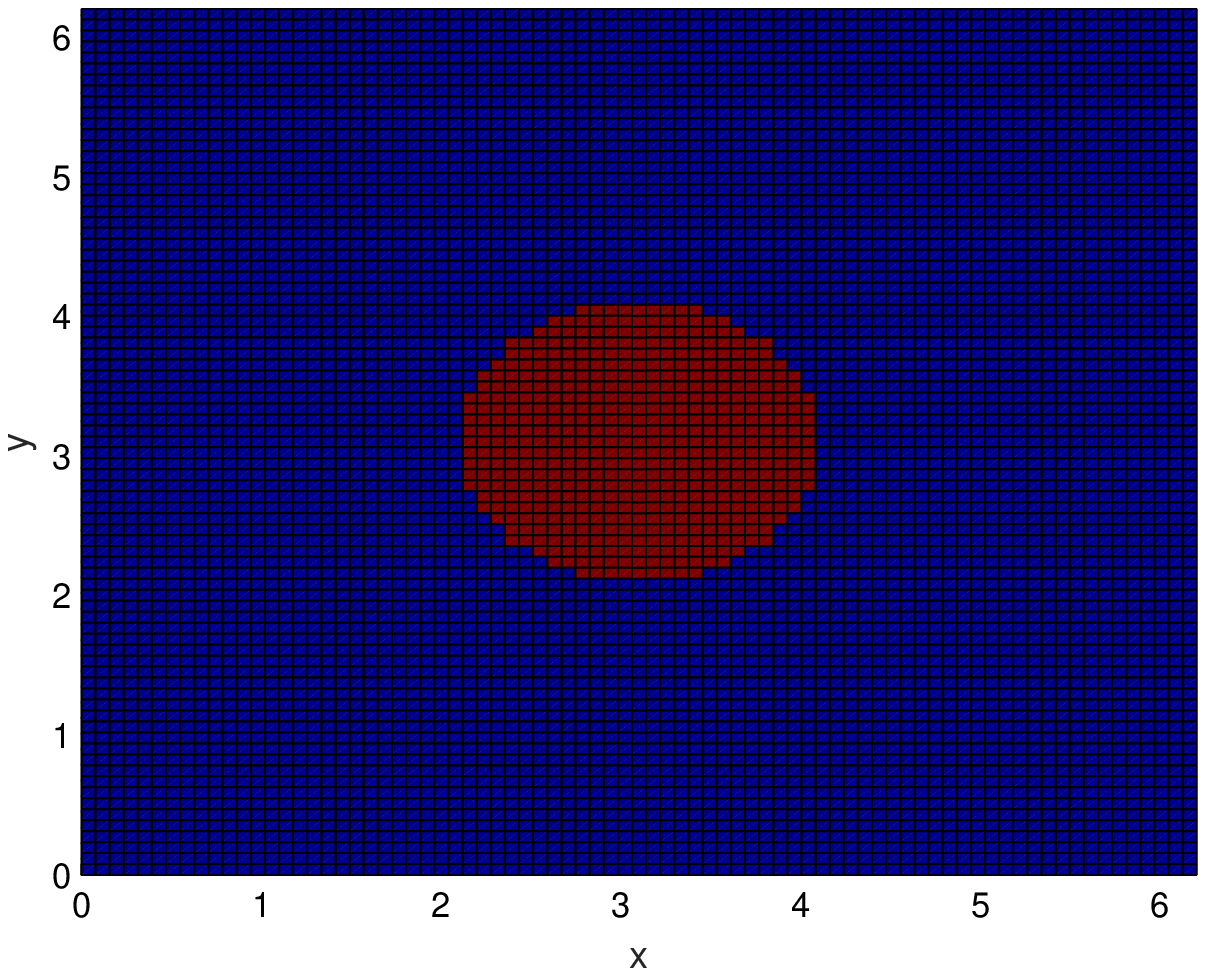}}
\subfigure[Initial value II, $t=0$. ]{
\label{fig2:b1} %% label for second subfigure
\includegraphics[height=3cm,width=3.5cm]{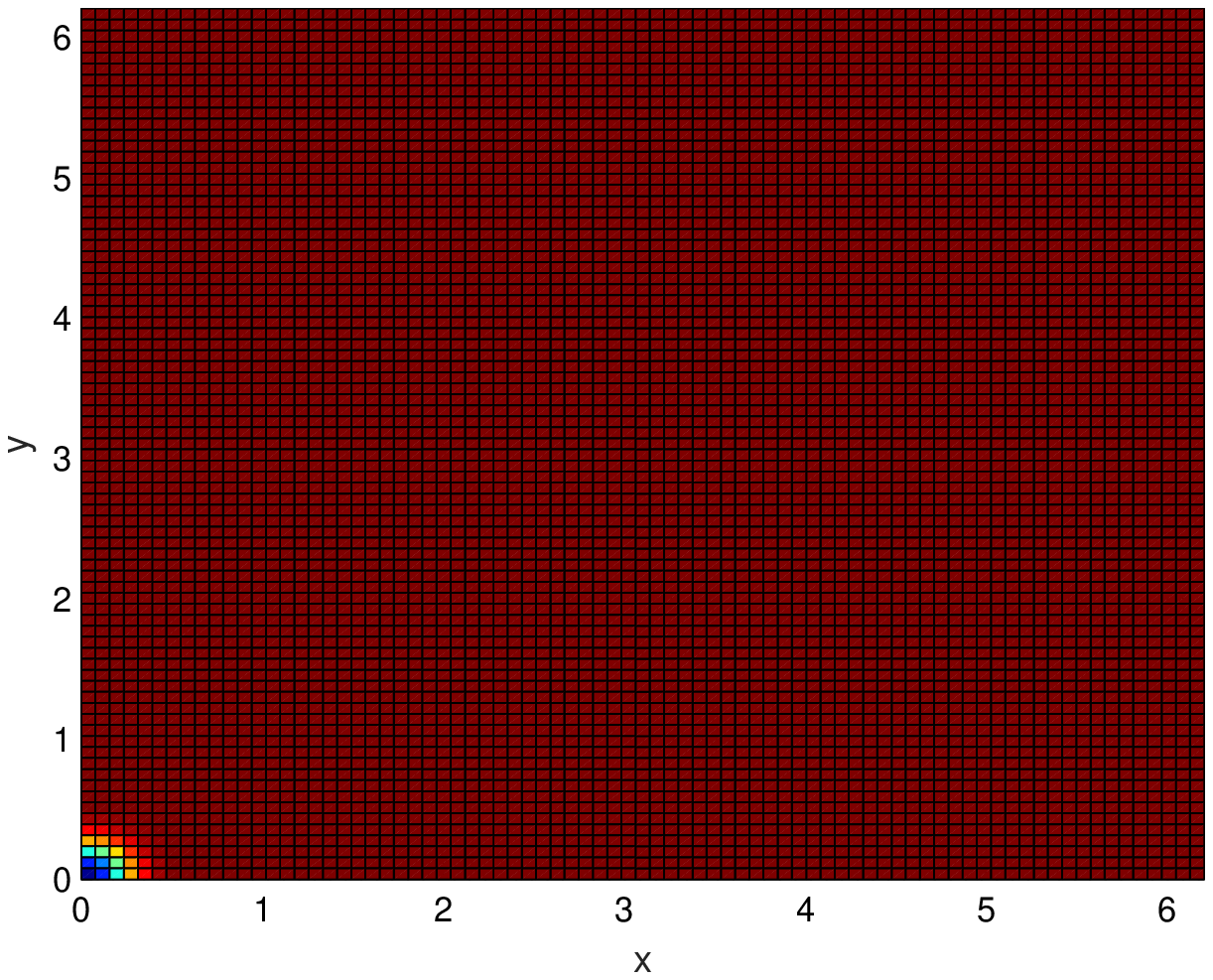}}
\subfigure[Initial value III, $t=0$.]{
\label{fig2:c1} %% label for first subfigure
\includegraphics[height=3cm,width=3.5cm]{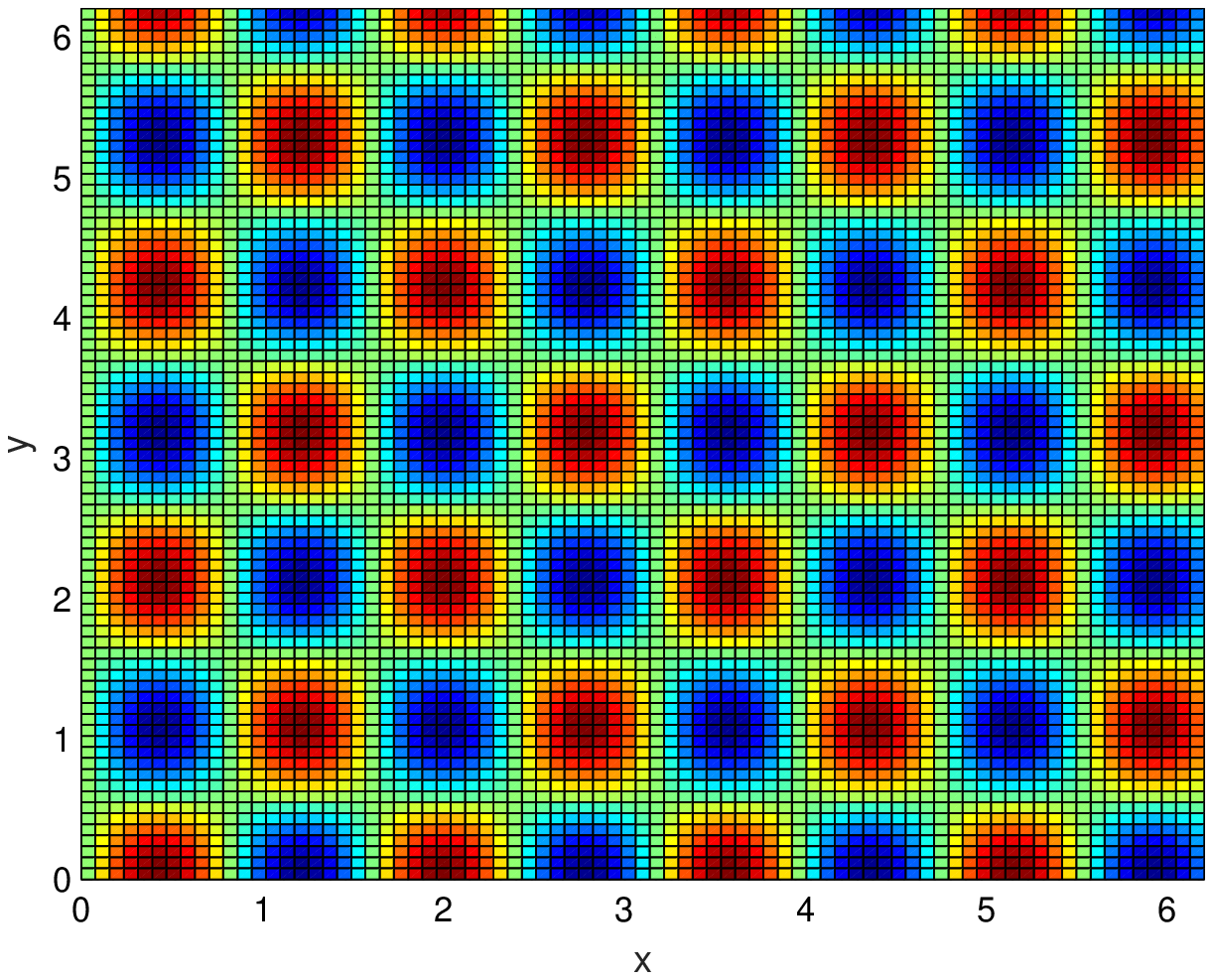}}
\subfigure[Initial value I, $t=10$. ]{
\label{fig2:a2} %% label for second subfigure
\includegraphics[height=3cm,width=3.5cm]{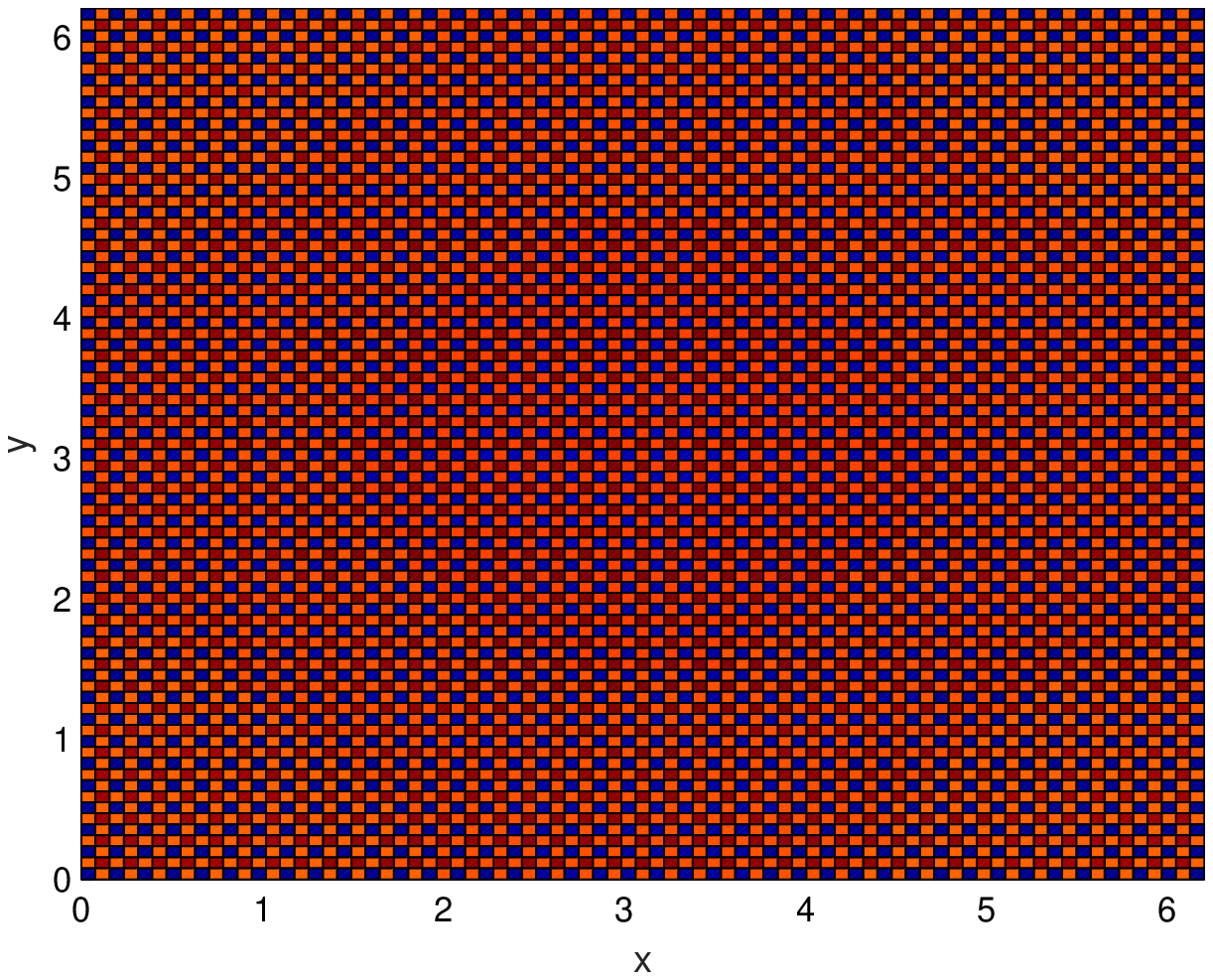}}
\subfigure[Initial value II, $t=10$.]{
\label{fig2:b2} %% label for first subfigure
\includegraphics[height=3cm,width=3.5cm]{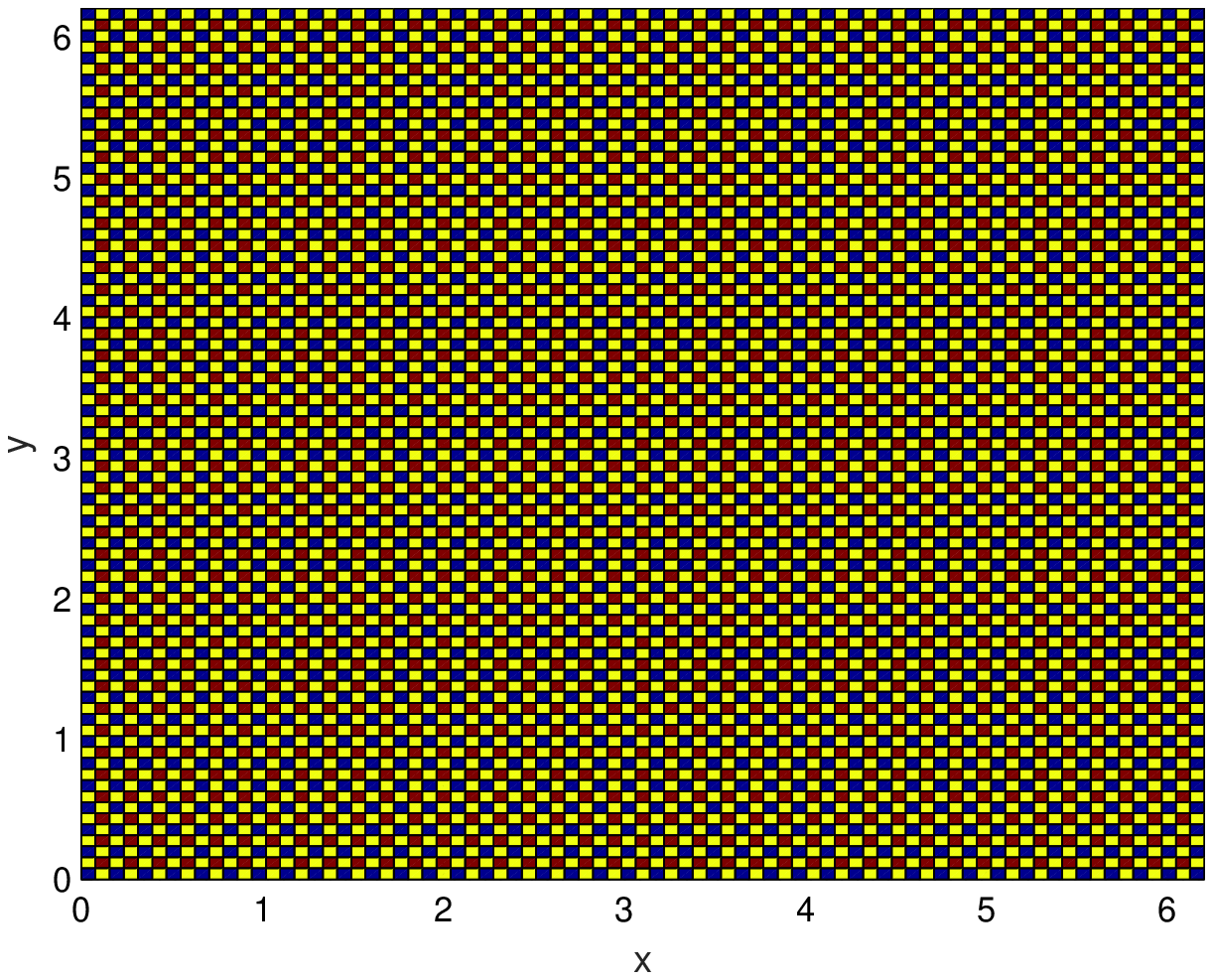}}
\subfigure[Initial value III, $t=10$. ]{
\label{fig2:c2} %% label for second subfigure
\includegraphics[height=3cm,width=3.5cm]{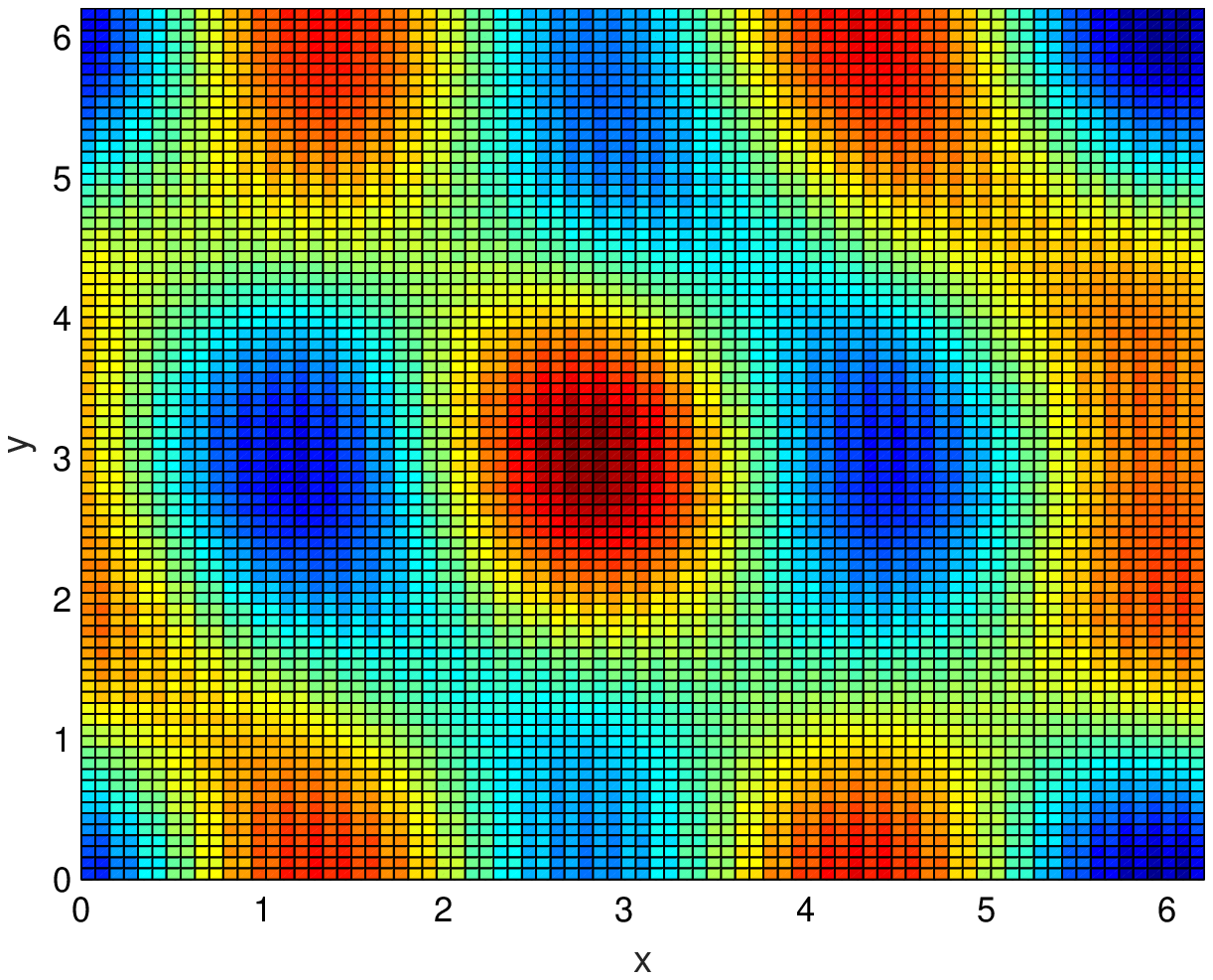}}
\subfigure[Initial value I, $t=50$.]{
\label{fig2:a3} %% label for first subfigure
\includegraphics[height=3cm,width=3.5cm]{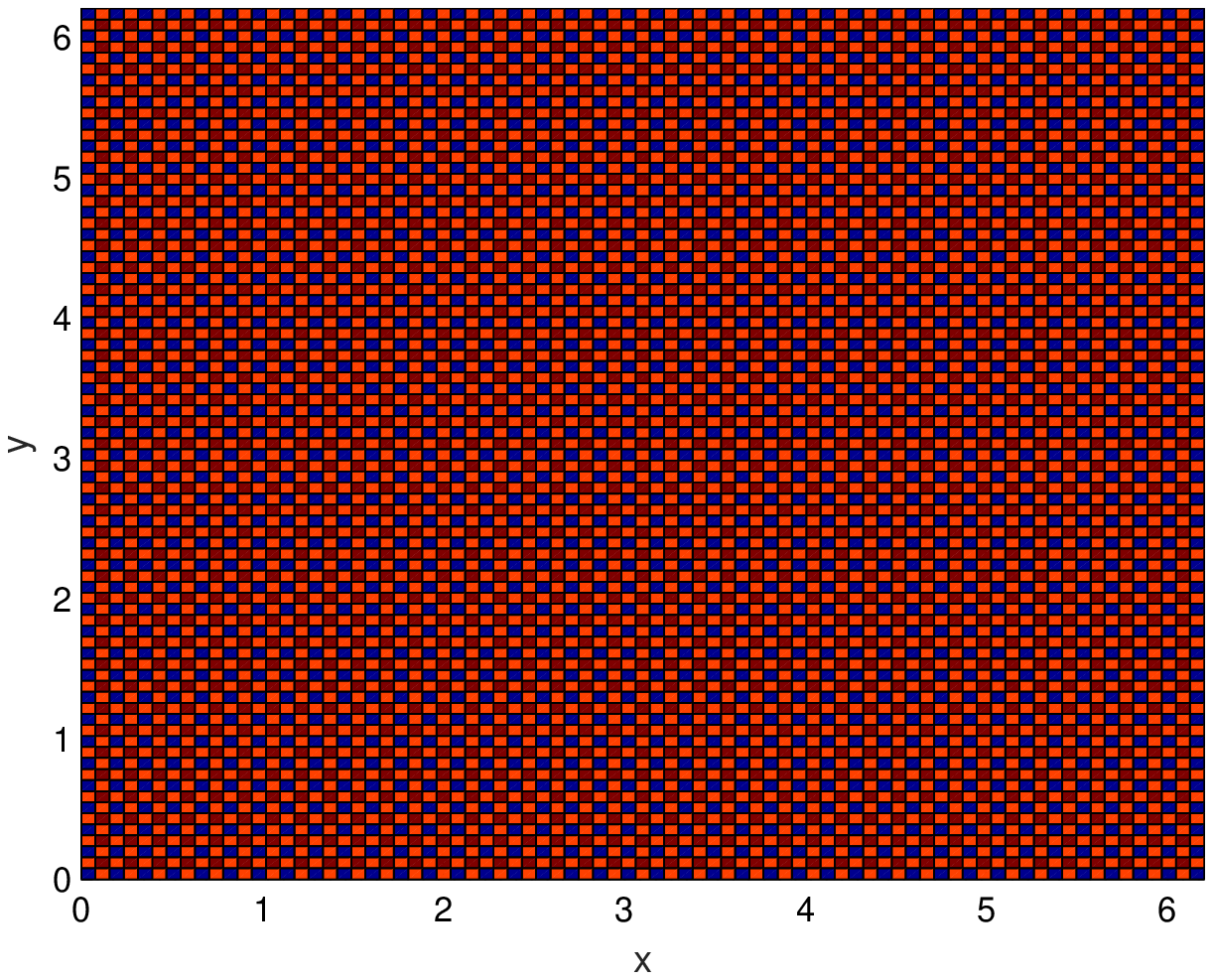}}
\subfigure[Initial value II, $t=50$. ]{
\label{fig2:b3} %% label for second subfigure
\includegraphics[height=3cm,width=3.5cm]{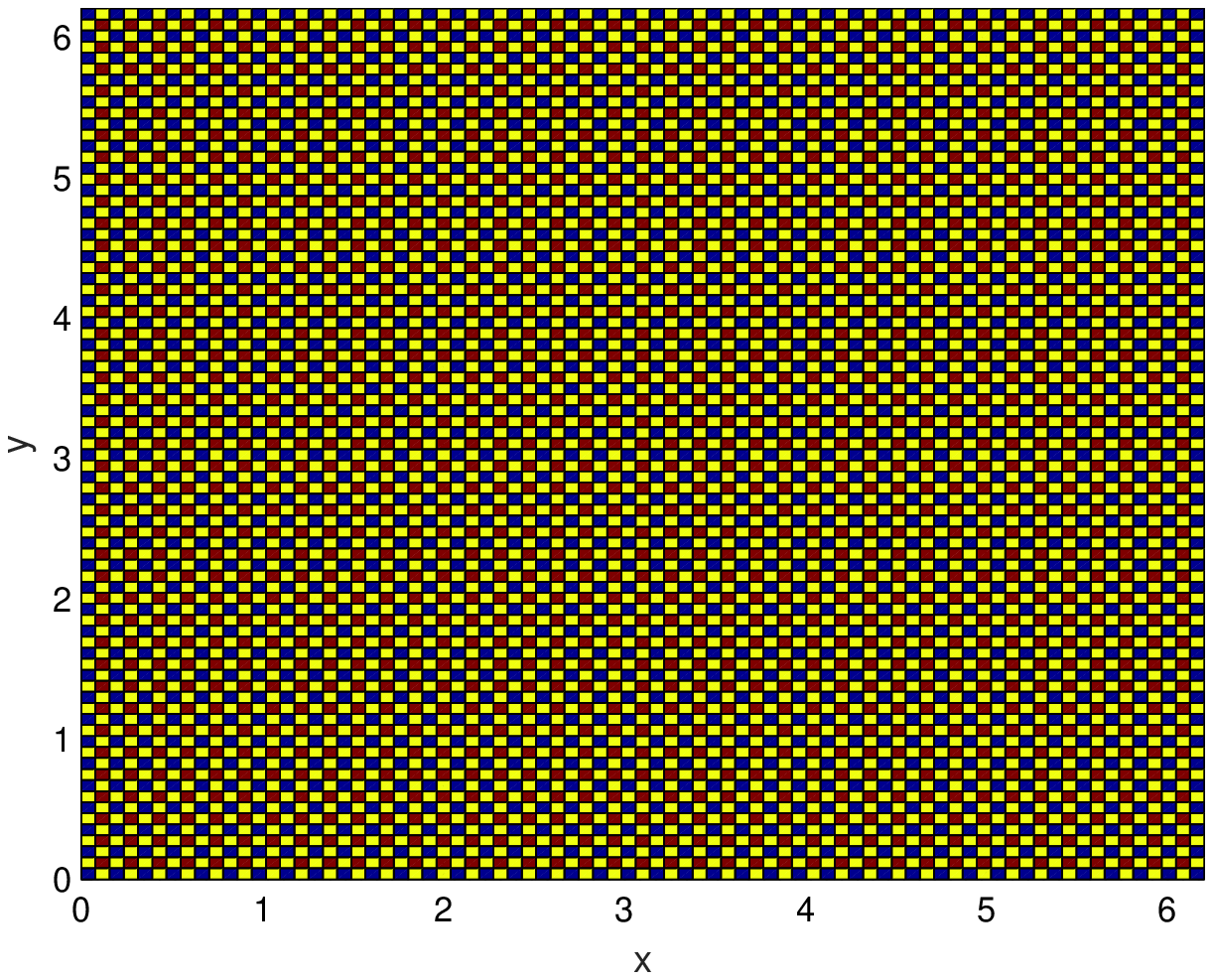}}
\subfigure[Initial value III, $t=50$. ]{
\label{fig2:c3} %% label for second subfigure
\includegraphics[height=3cm,width=3.5cm]{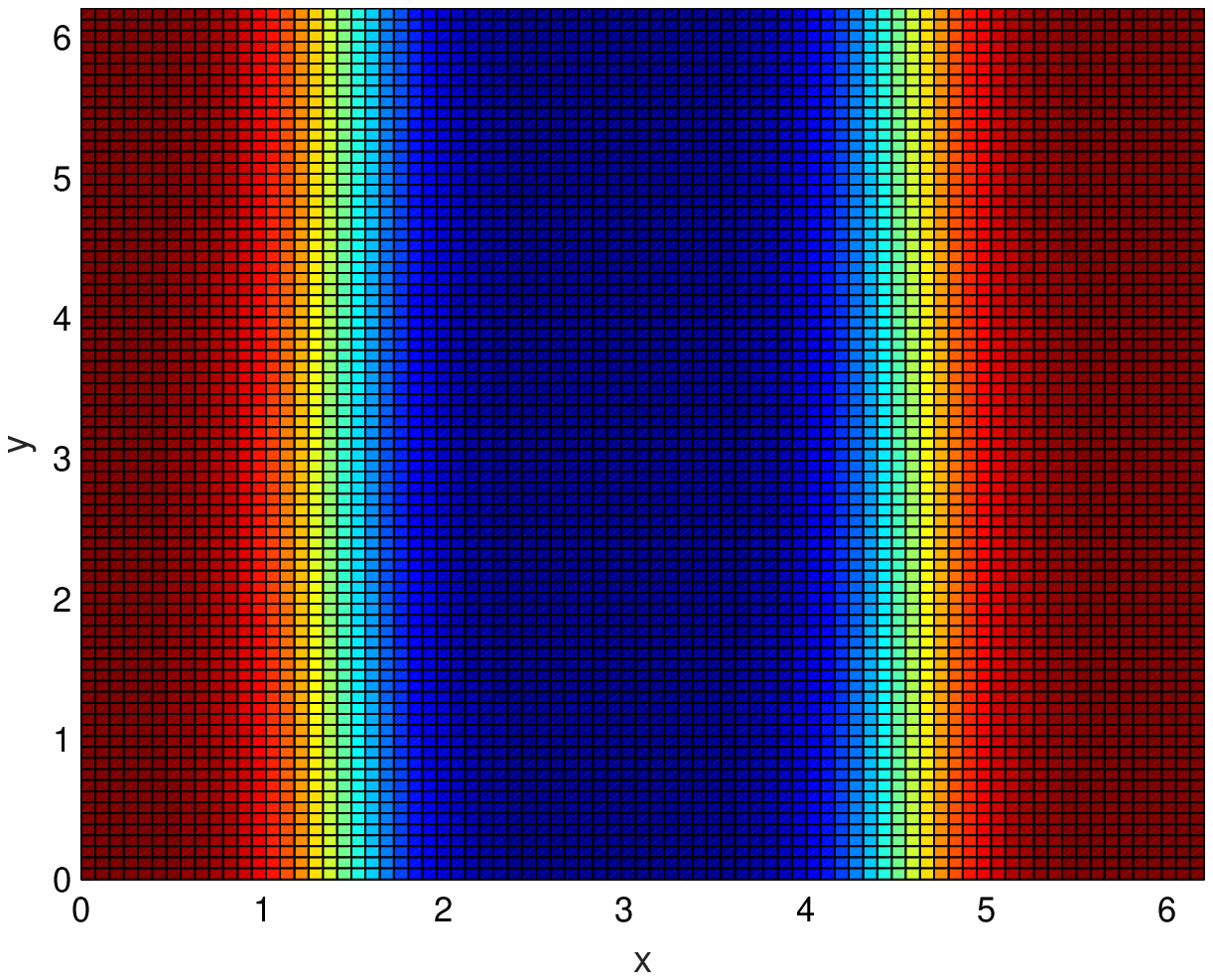}}
\subfigure[Initial value I, $t=100$.]{
\label{fig2:a4} %% label for first subfigure
\includegraphics[height=3cm,width=3.5cm]{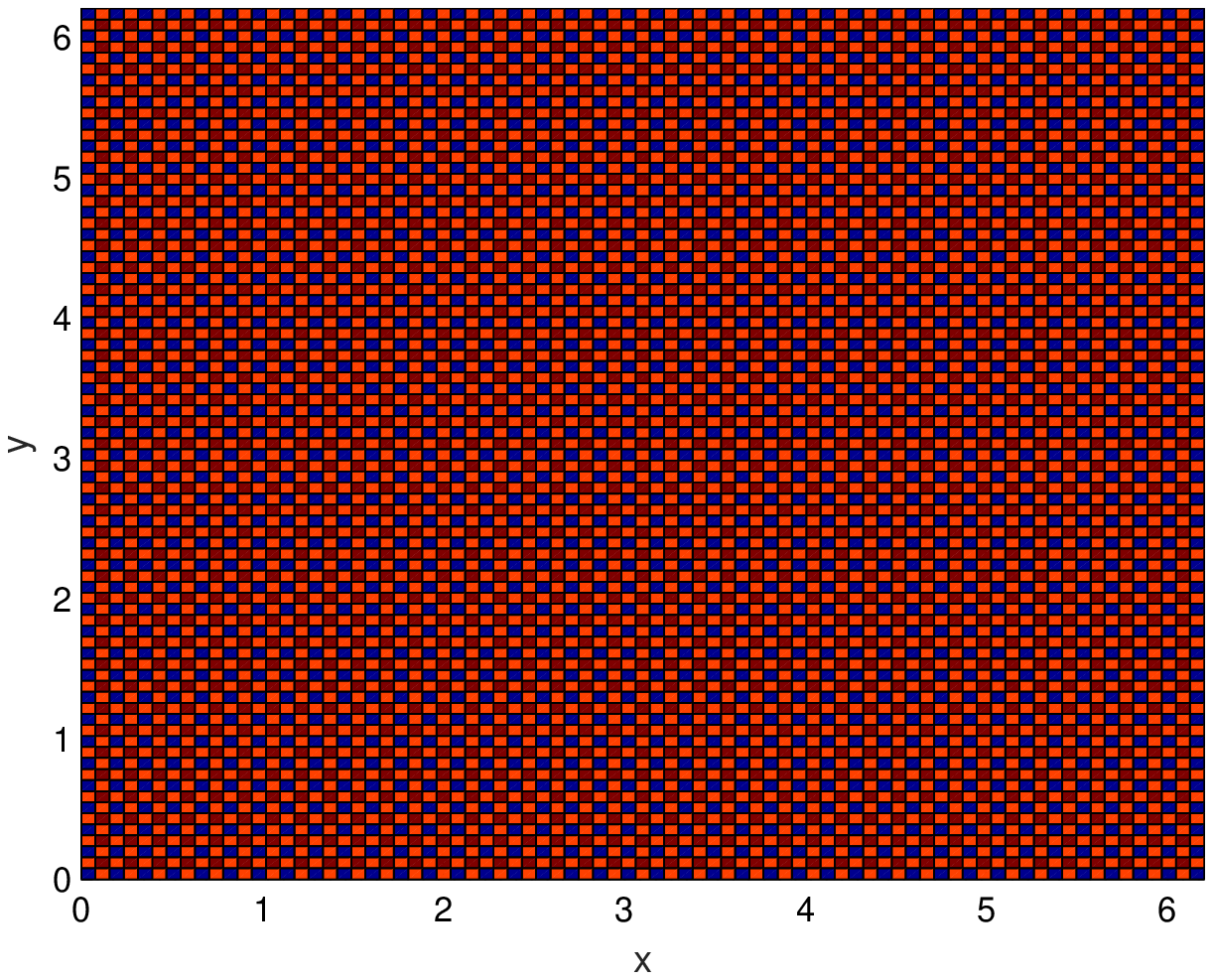}}
\subfigure[Initial value II, $t=100$. ]{
\label{fig2:b4} %% label for second subfigure
\includegraphics[height=3cm,width=3.5cm]{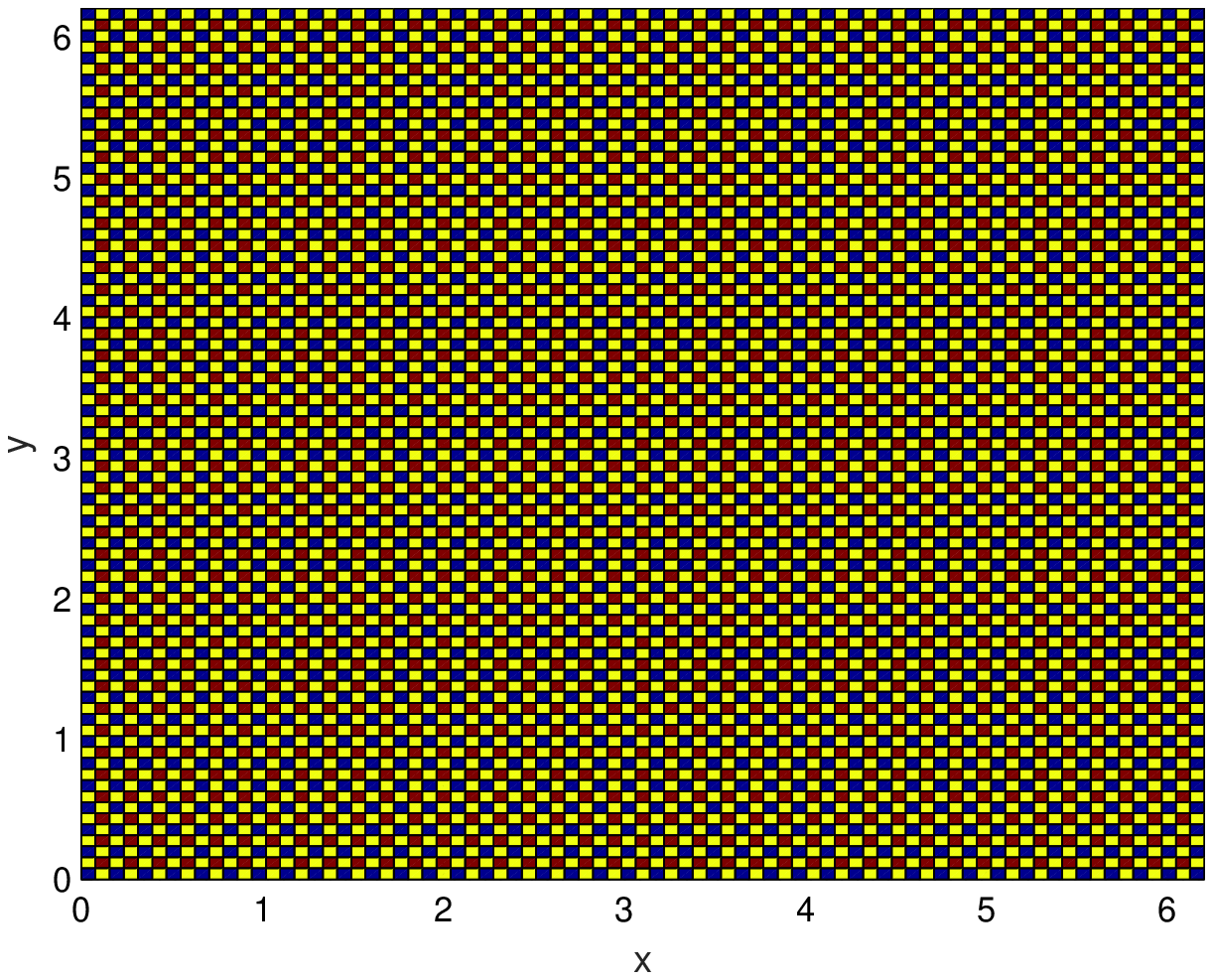}}
\subfigure[Initial value III, $t=100$. ]{
\label{fig2:c4} %% label for second subfigure
\includegraphics[height=3cm,width=3.5cm]{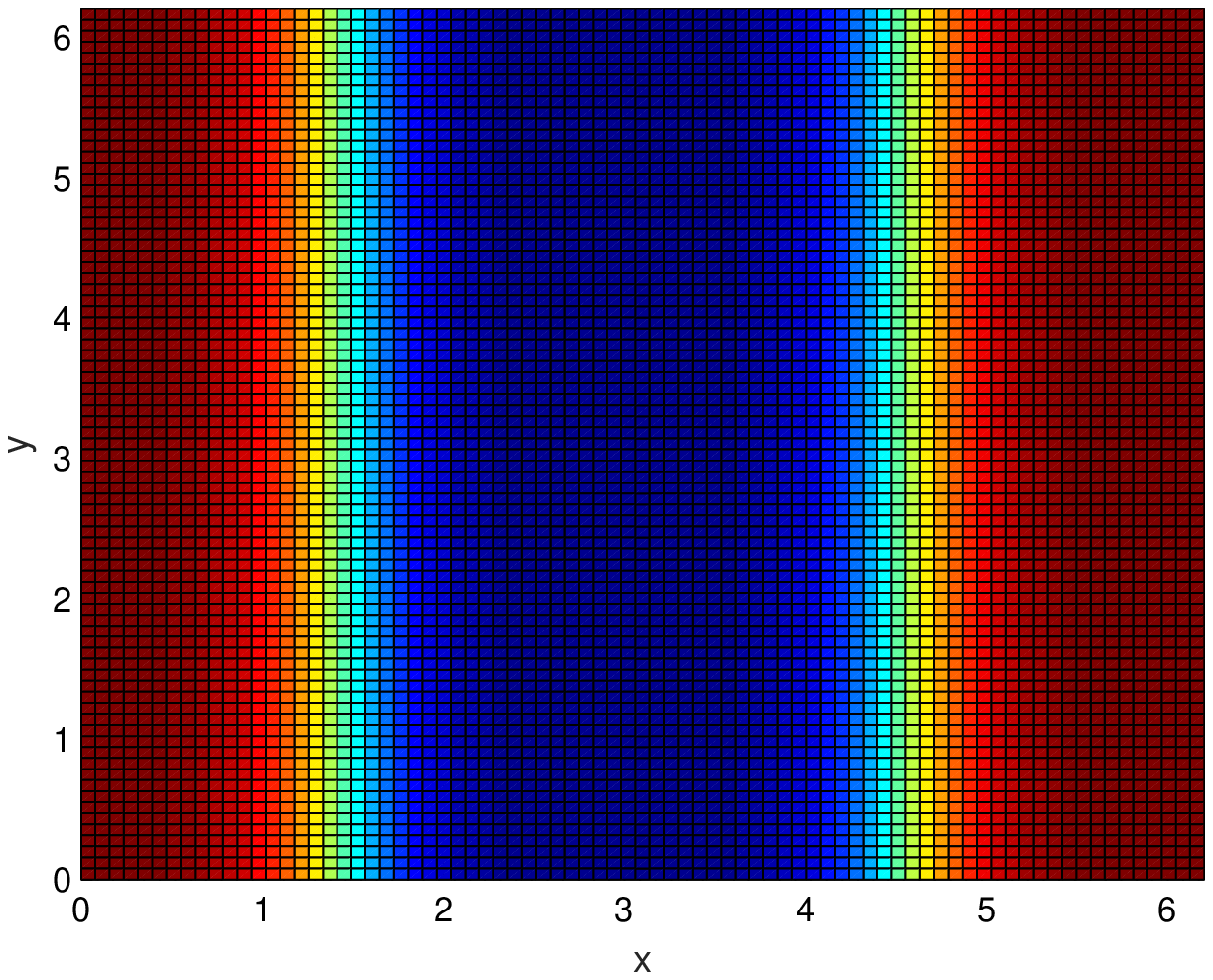}}
\caption{The Cahn-Hilliard equation with different initial values: plot of the numerical solutions at different $t$'s, where $k=0.1$, $\varepsilon=0.1$. }
\label{fig:2} %% label for entire figure
\centering
\end{figure}

\section{Conclusions} \label{5}
In this paper we focused on dissipative systems generated by Cahn-Hilliard equation and time discretization since long time approximation seems to be the key issue involved. We provide a long time numerical stability analysis for the implicit Euler scheme for Cahn-Hilliard equation with polynomial nonlinearity; uniform in time estimates for this scheme, in $H^{-1}$ and $H^s (s=1,2,3)$ norms, are derived. In addition, we show the uniform bound for the time semi-discrete chemical potential $\omega^{n+1}$ in $H^1$ norm with the aid of the variable steps-sizes uniform discrete Gronwall lemma. As a consequence of these estimates, for every time grid, we build a global attractor of the discrete-in-time dynamical system. The numerical results obtained by this scheme, combined with Fourier pseudospectral spatial approximation, have also verified such a long time stability.

%\section*{Acknowledgments}
%The author thanks the anonymous authors whose work largely
%constitutes this sample file. He also thanks the INFO-TeX mailing
%list for the valuable indirect assistance he received.

\end{document}